\documentclass[10pt, journal]{IEEEtran}

\usepackage{amsmath,mathtools,amssymb,graphicx,paralist,subfigure,pdfpages,cite,algorithm,algpseudocode,paralist,setspace}
\usepackage{verbatim}
\usepackage{algorithm}
\usepackage{amsthm}
\usepackage{verbatim}
\usepackage{algorithm}
\usepackage{algpseudocode}

\newtheorem{lem}{Lemma} 
\newtheorem{theorem}{Theorem}
\newtheorem{definition}{Definition}

\newtheorem{cor}{Corollary}
\newtheorem{rmk}{Remark}
\newtheorem{assump}{Assumption}

\def\R{\mathbb{R}}
\def\mc{\mathcal}
\def\mb{\mathbf}
\def\mbb{\mathbb}
\def\ra{\rightarrow}
\def\P{\mathbf{P}}

\def\CGD{\textbf{\texttt{GD}}}

\def\GDA{\textbf{\texttt{GT-GDA}}}
\def\GDAl{\textbf{\texttt{GT-GDA-Lite}}}
\def\DGDA{\textbf{\texttt{D-GDA}}}
\def\CGD{\textbf{\texttt{GD}}}
\def\GD{\textbf{\texttt{GDA}}}
\def\EG{\textbf{\texttt{EG}}}
\def\OGD{\textbf{\texttt{OGDA}}}

\IEEEoverridecommandlockouts

\def\R{\mathbb{R}}
\def\mc{\mathcal}
\def\mb{\mathbf}
\def\mbb{\mathbb}
\def\ra{\rightarrow}
\def\la{\leftarrow}
\def\P{\mathbf{P}}

\def\mt{\times}
\def\mbb{\mathbb}%R
\def\mb{\mathbf}%vector
\def\mc{\mathcal}%set
\def\wh{\widehat}
\def\wt{\widetilde}
\def\ol{\overline}

\def\bds{\boldsymbol}
\newcommand{\mn}[1]{{\left\vert\kern-0.25ex\left\vert\kern-0.25ex\left\vert\kern0.3ex #1 
		\kern0.3ex\right\vert\kern-0.25ex\right\vert\kern-0.25ex\right\vert}}

\IEEEoverridecommandlockouts
\begin{document}
\title{Distributed saddle point problems for strongly concave-convex functions}
\author{Muhammad I. Qureshi and Usman A. Khan\\Tufts University, Medford, MA
		\thanks{The authors are with the Electrical and Computer Engineering Department at Tufts University; \texttt{muhammad.qureshi@tufts.edu} and \texttt{khan@ece.tufts.edu}. This work has been partially supported by NSF under awards \#1903972 and \#1935555. 
		}
	}
\maketitle

\begin{abstract}
In this paper, we propose~$\GDA$, a distributed optimization method to solve saddle point problems of the form:~${\min_{\mb{x}} \max_{\mb{y}} \left\{ F(\mb x,\mb y) :=G(\mb x) + \langle \mb y, \ol{P} \mb x \rangle - H(\mb y) \right\}}$, where the functions~$G(\cdot)$,~$H(\cdot)$, and the the coupling matrix~$\ol{P}$ are distributed over a strongly connected network of nodes.~$\GDA$ is a first-order method that uses gradient tracking to eliminate the dissimilarity caused by heterogeneous data distribution among the nodes. In the most general form,~$\GDA$ includes a consensus over the local coupling matrices to achieve the optimal (unique) saddle point, however, at the expense of increased communication. To avoid this, we propose a more efficient variant~$\GDAl$ that does not incur the additional communication and analyze its convergence in various scenarios. We show that~$\GDA$ converges linearly to the unique saddle point solution when~$G$ is smooth and convex,~$H$ is smooth and strongly convex, and the global coupling matrix~$\ol{P}$ has full column rank. We further characterize the regime under which~$\GDA$ exhibits a network topology-independent convergence behavior. We next show the linear convergence of~$\GDAl$ to an error around the unique saddle point, which goes to zero when the coupling cost~${\langle \mb y, \ol{P} \mb x \rangle}$ is common to all nodes, or when~$G$ and~$H$ are quadratic. Numerical experiments illustrate the convergence properties and importance of~$\GDA$ and~$\GDAl$ for several applications.

\end{abstract}

\section{Introduction}\label{intro}
Saddle point or min-max problems are of significant practical value in many signal processing and machine learning applications~\cite{hall, golub, Goodfellow, sinha, jordan, jordan_GDA, stokes}. Applications of interest include but are not limited to constrained and robust optimization, weighted linear regression, and reinforcement learning. In contrast to the traditional minimization problems where the goal is to find a global (or a local) minimum, the objective in saddle point problems is to find a point that maximizes the function value in one direction and minimizes in the other. Consider for example Fig.~\ref{f1} (left), where we show a simple function landscape~(${F:\mbb R^2\ra\mbb R}$) that increases in one direction and decreases in the other. Examples of such functions appear in constrained optimization where adding the constraints as a Lagrangian naturally leads to saddle point formulations. 

\begin{figure}[!h]
\centering
\subfigure{\includegraphics[width=1.7in]{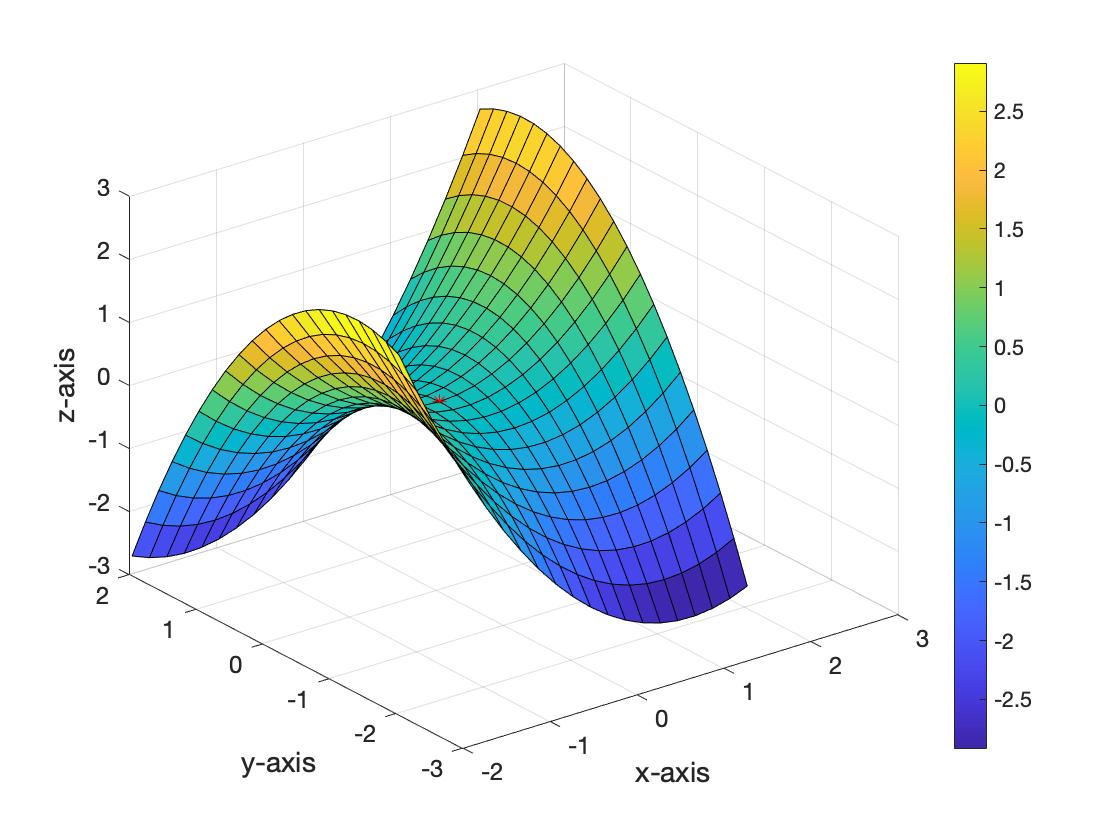}}
\subfigure{\includegraphics[width=1.7in]{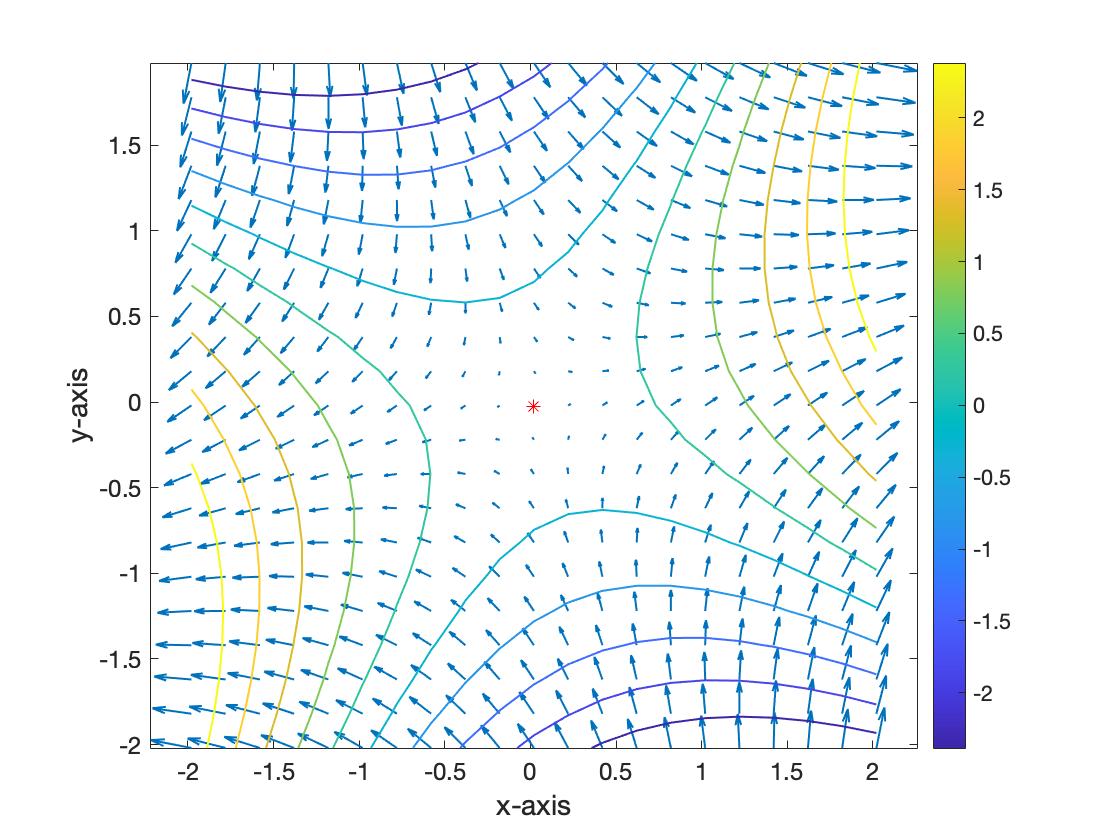}}
\caption{Plot of two dimensional strongly concave-convex saddle point problem (left) and the corresponding gradient directions (right). The curves represent contours to show the value of the function and the red star is the point where partial gradients are all~$0$.}
\label{f1}
\end{figure}

Gradient descent ascent methods are popular approaches towards saddle point problems. To find a saddle point of the function in Fig.~\ref{f1} (left), we would like to maximize~$F$ with respect to the corresponding variable, say~$\mb{y}$, and minimize~$F$ in the direction, say~$\mb{x}$. A natural way is to compute the partial gradients~$\nabla_{\mb{y}} F$ and~$\nabla_{\mb{x}} F$, shown in Fig.~\ref{f1} (right). Then update the~$\mb{y}$ estimate moving in the direction of~$\nabla_{\mb{y}} F$ and the~$\mb{x}$ estimate moving opposite to the direction of~$\nabla_{\mb{x}} F$. The arrows, shown in Fig.~\ref{step}, point towards the next step of~$\GD$ dynamics and the method converges to the unique saddle point (red star) under appropriate conditions on~$F$. The extension of this method for convex and strongly concave, and strongly convex and concave objectives is straightforward as it is intuitive that the saddle point~${(\mb{x}^*, \mb{y}^*) \in {\R^{p_x} \mt \R^{p_y}}}$ is unique such that~${\forall \mb{x} \in \R^{p_x}}$ and~${\forall \mb{y} \in \R^{p_y}}$,
\begin{align*}
    F(\mb{x}^*, \mb{y}) \leq F(\mb{x}^*, \mb{y}^*) \leq F(\mb{x}, \mb{y}^*).
\end{align*}

\begin{figure}[!b]
\centering
\subfigure{\includegraphics[width=2.5in]{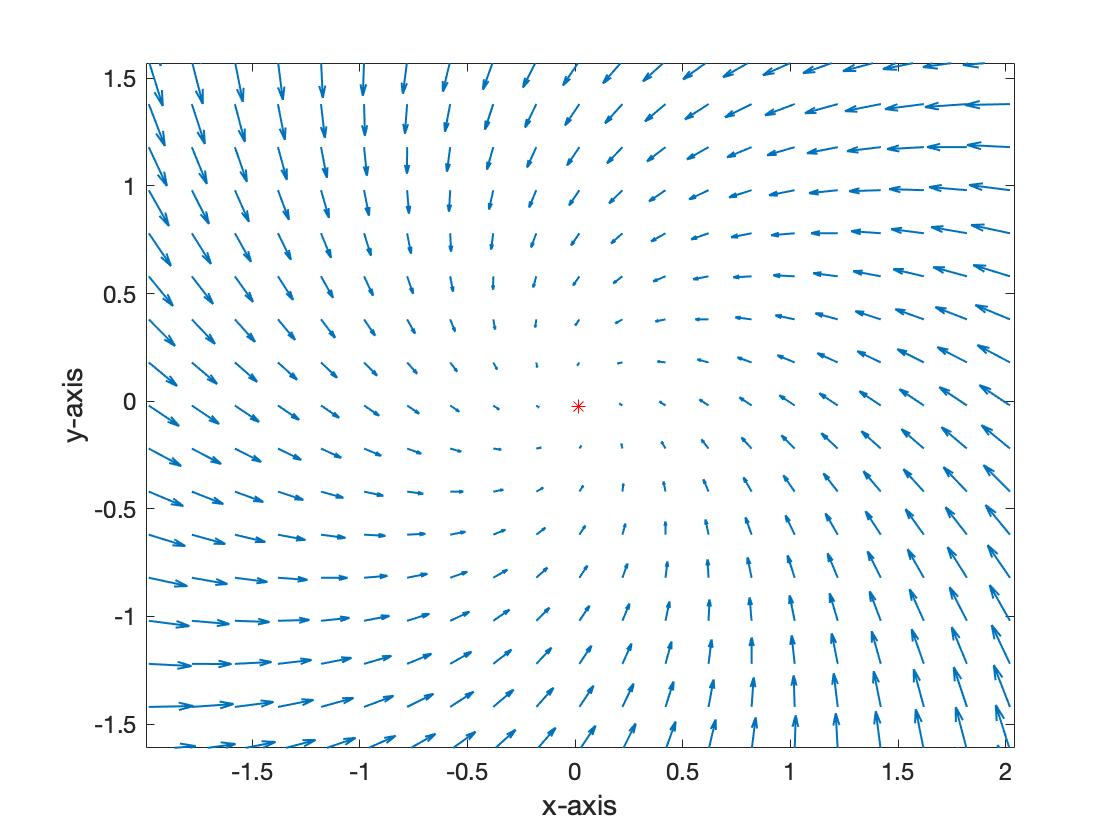}}
\caption{The arrows point towards the
next step of the gradient descent ascent dynamics. The unique saddle point is denoted by the red star.}
\label{step}
\end{figure}

The traditional approaches mentioned above assume that the entire dataset is available at a central location. In many modern applications~\cite{forero2010consensus,6119236,tutorial_nedich,DOPT_survey_yang}, however, data is often collected  by a large number of geographically distributed devices or nodes and communicating/storing the entire dataset at a central location is practically infeasible. Distributed optimization methods are often preferred in such scenarios, which operate by keeping data local to each individual device and exploit local computation and communication to solve the underlying problem. Such methods often deploy two types of computational architectures: (i) master/worker networks -- where the data is split among multiple workers and computations are coordinated by a master (or a parameter server); (ii) peer-to-peer mesh networks -- where the nodes are able to communicate only with nearby nodes over a strongly connected network topology. The topology of mesh networks is more general as it is not limited to hierarchical master/worker architectures. 

In this paper, we are interested in solving distributed saddle point optimization problems over peer-to-peer networks, where the corresponding data and cost functions are distributed among~$n$ nodes, communicating over a strongly connected weight-balanced directed graph. In this formulation, the networked nodes are tasked to find the saddle point of a sum of local cost functions~$f_i(\mb{x}, \mb{y})$, where~$\mb{x} \in \R^{p_x}$ and~$\mb{y}\in \R^{p_y}$. Mathematically, we consider the following problem:
\begin{align*}
    \P: \min_{\mb{x}\in \R^{p_x}} \max_{\mb{y}\in \R^{p_y}} F(\mb{x}, \mb{y}) &= \min_{\mb{x}\in \R^{p_x}} \max_{\mb{y}\in \R^{p_y}} \frac{1}{n} \sum_{i=1}^{n} f_i(\mb{x}, \mb{y}),
\end{align*}
where each local cost~$f_i(\mb{x}, \mb{y})$ is private to node~$i$ and takes the form as follows
\[
f_i(\mb{x}, \mb{y}) :=g_i(\mb x) + \langle \mb y, P_i \mb x \rangle - h_i(\mb y).
\]
We assume that~${G(\mb{x}):=\frac{1}{n} \sum_{i=1}^n g_i(\mb{x})}$ is convex and ${H(\mb{y}):=\frac{1}{n} \sum_{i=1}^n h_i(\mb{y})}$ is strongly convex\footnote{Note that the problem class~$\P$ includes~$-H$, which is strongly concave.}, while the global coupling matrix~${\ol{P} = \frac{1}{n} \sum_{i=1}^n P_i \in \R^{p_y \mt p_x}}$ has full column rank. Such problems arise naturally in many subfields of signal processing, machine learning, and statistics~\cite{du_quad, du,WW_scutari, xianWWSGDA}. 

\subsection{Related work}
Theoretical studies on solutions for saddle point problems, in centralized scenarios, have attracted significant research~\cite{Von_Neumann, hall, Basar_Olsder, golub}. Recently, saddle point or min-max problems have become increasingly relevant because of their applications in constrained and robust optimization, supervised and unsupervised learning, image reconstruction, and reinforcement learning~\cite{Goodfellow, sinha, jordan, stokes}. Commonly studied sub-classes of saddle point problems of the form~$\mb P$ are when~$G(\cdot)$ and~$H(\cdot)$ are assumed to be quadratic~\cite{du_quad} or strongly convex~\cite{mokhtari}. In~\cite{mokhtari}, the authors proposed a unified analysis technique for extra-gradient ($\EG$) and optimistic gradient descent ascent ($\OGD$) methods assuming strongly concave-strongly convex saddle point problems. Furthermore, they discussed the convergence rates for the underlying problem classes and bilinear objective functions. A more general approach was taken in~\cite{du}, where~$G(\cdot)$ was considered convex but not strongly convex. In \cite{OGDA_contact}, the authors established the convergence of~$\OGD$ only assuming the existence of saddle-points. These are first-order methods with some modification of the vanilla gradient descent~($\CGD$). Apart from gradient-based methods, zeroth-order optimization techniques are proposed in~\cite{gdfree_VPoor,bo,ga,pso}. Such methods are useful when gradient computation is not feasible either because the objective function is unknown and the partial derivatives cannot be evaluated for the whole search space, or the evaluation of partial gradients is too expensive. In such cases, Bayesian optimization~\cite{bo} or genetic algorithms~\cite{ga,pso} are used to converge to the saddle points. These techniques are usually slower than gradient-based methods. 

When the data is distributed over a network of nodes, existing work has mainly focused on \textit{minimization} problems~\cite{DSGD_nedich, DGD_Kar, Diffusion_Chen,DEXTRA,add-opt,AB,ABN,proxPDA_Hong,ADMM_hong,SUCAG,Async_NN_Wei,dual_optimal_ICML}. Of significant relevance are distributed methods that assume access to a  first-order oracle where the early work includes~\cite{DSGD_nedich, DGD_Kar, Diffusion_Chen}. The performance of these methods is however limited due to their inability to handle the dissimilarity between local and global cost functions, i.e.,~$\nabla f_i \neq \nabla F$. In other words, linear convergence is only guaranteed but to an inexact solution (with a constant stepsize). To avoid this inaccuracy while keeping linear convergence, recent work~\cite{GT_CDC,NEXT,DSGD_NIPS,D2,harnessing,AB,ABN,AB_proc} propose a gradient tracking technique that allows each node to estimate the global gradient with only local communications. Of note are also distributed stochastic problems where gradient tracking is combined with variance reduction to achieve state-of-the-art results for several different classes of problems~\cite{GT_SAGA_SPM,D2,DAVRG,pushsaga,GT_SARAH_ncvx,GT_HSGD_ncvx,xin2021stochastic}.

On saddle point problems, there is not much progress made towards distributed solutions. Recent work in this regard includes~\cite{WW_GDA_mingyi,  WW_scutari, xianWWSGDA, Local_GDA, AB_GDA} but majority of them do not consider heterogeneous data among different nodes. To deal with the dissimilarity between the local and global costs,~\cite{WW_scutari} develops a function similarity metric. Similarly,~\cite{Local_GDA} proposes local stochastic gradient descent ascent using similarity parameters but it is restricted to master/worker networks that are typical in federated learning scenarios. To get rid of the aforementioned similarity assumptions,~\cite{WW_GDA_mingyi} uses gradient tracking to eliminate this dissimilarity but assumes the functions~$G(\cdot)$ and~$H(\cdot)$ to be quadratic. Similarly,~\cite{AB_GDA} extends~\cite{WW_GDA_mingyi} to directed graphs and show linear convergence for quadratic functions. 

\subsection{Main contributions}
In this paper, we propose~$\GDA$ and~$\GDAl$ to solve the underlying distributed saddle point problem~$\P$. The~$\GDA$ algorithm performs a gradient descent in the~$\mb x$ direction and a gradient ascent in the~$\mb y$ direction, both of which are combined with a network consensus term along with the communication of coupling matrices~$P_i$ with neighbours.~$\GDAl$ is a lighter (communication-efficient) version of~$\GDA$, which does not require consensus over the coupling matrices and therefore reduces the communication complexity. To address the challenge that arises due to the dissimilarity between the local and global costs, the proposed methods use gradient tracking in both of the descend and ascend updates. To the best of our knowledge, there is no existing work for Problem~$\P$ that shows linear convergence when~$G(\cdot)$ is convex and~$H(\cdot)$ is strongly convex. The main contributions of this paper are described next: 
% \begin{inparaenum}[(i)]
% \item We propose a novel algorithm~$\GDA$ to solve decentralized saddle point problems using gradient tracking, on both ascent and descent steps, to address data dissimilarity across the nodes.
% \item We show that~$\GDA$ converges linearly to the global optimizer~$(\mb{x}^*, \mb{y}^*)$ and provide explicit expressions for the gradient computation complexities.
% \item We consider the problem class to be strongly-concave and convex which hasn't been considered before to the best of authors' knowledge.
% \end{inparaenum}

\textbf{Novel Algorithm.} We propose a novel algorithm that uses gradient tracking for distributed gradient descent ascent updates. Gradient tracking implements an extra consensus update where the networked nodes track the global gradients with the help of local information exchange among the nearby nodes. %This procedure is based on the dynamic average consensus protocol~\cite{DAC} that was developed to track the average of time-varying functions. 

% For centralized settings, all the data is available at a single node which is used to compute the gradients at each step. In decentralized scenario, data is often heterogeneous and is distributed among different nodes. Each node can access it's local data set. Thus, making the gradient computations inaccurate. Our goal is to optimize the average of all local objective functions~$f_i$. We use gradient tracking to avoid this error caused by dissimilarity across the nodes.

\textbf{Weaker assumptions.}
We consider the problem class~$\P$ such that~$g_i$ and~$h_i$ are smooth,~$G$ is convex,~$H$ is strongly convex and the coupling matrix~$\ol{P}$ has full column rank. We note that the constituent local functions,~$g_i(\cdot)$ and~$h_i(\cdot)$, can be non-convex as we only require convexity on their average. Earlier work~\cite{WW_GDA_mingyi} that shows linear convergence of distributed saddle point problems is only applicable to specific quadratic functions~$G$ and~$H$, used in reinforcement learning, and does not provide explicit rates. 
% Similarly,~\cite{AB_GDA} extends the quadratic case to directed graphs. Moreover, 
It is noteworthy that the proposed problem~$\P$ can be written in the primal form as follows:
\begin{align*}
    \min_{\mb{x}} \theta(\mb{x}) = \min_{\mb{x}} \left\{ H^*(\ol{P} \mb{x}) + G(\mb{x}) \right \}
\end{align*}
where~$H^*(\cdot)$ is the conjugate function~\cite{rockafellar}, see Definition~\ref{conj_def}, of~$H(\cdot)$. We note that because~$G(\cdot)$ is strongly convex, it is enough to ensure that~$\ol{P}$ has full column rank to conclude that~$\theta(\cdot)$ is strongly convex~\cite{Nesterov_book}. This results in significantly weaker assumptions as compared to the available literature.

\textbf{Linear convergence and explicit rates.} We show that~$\GDA$ converges linearly to the unique saddle point~${(\mb{x}^*, \mb{y}^*)}$ of Problem~$\P$ under the assumptions described above. We note that all these assumptions are necessary for linear convergence even for the centralized case~\cite{du}. Furthermore, we evaluate explicit rates for gradient complexity per iteration and provide a regime in which the convergence of~$\GDA$ is network-independent. We also show linear convergence of~$\GDAl$ in three different scenarios and establish that the rate is the same as~$\GDA$ (potentially with a steady-state error) with reduced communication complexity.

\textbf{Exact analysis for quadratic problems.} We provide exact analytic expressions to develop the convergence characteristics of~$\GDAl$ when~$G$ and~$H$ are in general quadratic forms. With the help of matrix perturbation theory for semi-simple eigenvalues, we show that~$\GDAl$ converges linearly to the unique saddle point of the underlying problem.
\vspace{-0.5cm}
\subsection{Notation and paper organization}
We use lowercase letters to denote scalars, lowercase bold letters to denote vectors, and uppercase letters to denote matrices. We define~$\mb{0}_n$ as vector of~$n$ zeros and~$I_n$ as the identity matrix of~${n \mt n}$ dimensions.
For a function~${F(\mb x, \mb y)}$,~$\nabla_{\mb x}F$ is the gradient of~$F$ with respect to~$\mb x$, while~$\nabla_{\mb y}F$ is the gradient of~$F$ with respect to~$\mb y$. We denote the vector two-norm as~$\|\cdot\|$ and the spectral norm of a matrix induced by this vector norm as~$\mn{\cdot}$. We denote the weighted vector norm of a vector~$\mb{z}$ with respect to a matrix~$C$ as~${\|\mb{z}\|_{C}:=\mb{z}^\top C \mb{z}}$ and the spectral radius of~$C$ as~$\rho(C)$. We consider~$n$ nodes interacting over a potentially directed (balanced) graph~${\mc{G}=\{\mc{V,E}\}}$, where~${\mc V := \{1,\ldots,n\}}$ is the set of node indices, and~${\mc E\subseteq \mc V\times \mc V}$ is a collection of ordered pairs~$(i,r)$ such that node~$r$ can send information to node~$i$, i.e.,~${i\leftarrow r}$.

The rest of the paper is organized as follows. Section~\ref{mot_ad} provides the motivation, with the help of several examples, and describes the algorithms~$\GDA$ and~$\GDAl$. We discuss our main results in Section~\ref{main_res}, provide simulations in Section~\ref{sims}, the convergence analysis in Section~\ref{conv_ana}, and conclude the paper with Section~\ref{conc}.

\section{Motivation and Algorithm
Description} \label{mot_ad}
In this section, we provide some motivating applications that take the form of convex-concave saddle point problems.
\vspace{-0.5cm}
\subsection{Some useful examples}
\textbf{Distributed constrained optimization.} 
Minimizing an objective function under certain constraints is a fundamental requirement for several applications. For equality constraints, such problems can be written as:
\begin{align} \label{const_opt}
    \min_{\mb{x}} G(\mb{x}), \quad    \text{subject to} \quad \ol{P} \mb{x} = \mb{b},
\end{align}
which has a saddle point equivalent form written using the Lagrangian multipliers~$\mb{y}$:
\begin{align*}
    \mc{L}(\mb{x}, \mb{y}) &= G(\mb{x}) + \mb{y}^\top (\ol{P} \mb{x} - \mb{b}) \\
    &= G(\mb{x}) + \mb{y}^\top \ol{P} \mb{x} - \mb{y}^\top \mb{b}.
\end{align*}
Any solution of~\eqref{const_opt} is a saddle point~$(\mb{x}^*, \mb{y}^*)$ of the Lagrangian. Hence, it is sufficient to solve for
\begin{align*}
    \mc{L}(\mb{x}^*, \mb{y}^*) = \min_\mb{x} \max_\mb{y} \mc{L}(\mb{x}, \mb{y}).
\end{align*}
For large-scale problems, the data is distributed heterogeneously and each node possesses its local~$g_i(\cdot), P_i$ and~$\mb{b}_i$. The network aims to solve~\eqref{const_opt} such that~
\begin{equation*}
   G(\mb{x}) := \frac{1}{n} \sum_{i=1}^n g_i(\mb{x}), \quad \ol{P} := \frac{1}{n} \sum_{i=1}^n P_i, \quad \mb{b} := \frac{1}{n} \sum_{i=1}^n \mb{b}_i. 
\end{equation*}
% $${$$ 
Then for~${h_i(\mb{y}):= \langle \mb{b}_i, \mb{y} \rangle}$ and~${H(\mb{y}) := \frac{1}{n} \sum_{i=1}^n h_i(\mb{y})}$,~\eqref{const_opt} takes the same form as Problem~$\P$.

% \textbf{Robust Optimization.}
% {\iq In many practical applications, we observe} uncertainty in data while minimizing an objective function~$F$. The probability distribution of data and uncertainty is usually parameterized as~$\mbb{P}(\mb{y})$ and~$\psi$ respectively. Our aim is to optimize the following:
% \begin{align} \label{ro}
%     \min_\mb{x} \max_\mb{y} \mbb{E}_{\psi \sim \mbb{P}(\mb{y})} \left\{ F(\mb{x}, \psi)\right \}.
% \end{align}
% For several cases~\cite{rob}, the objective function of above problem has similar form as~$\P$.

\textbf{Distributed weighted linear regression and reinforcement learning.} Most applications of weighted linear regression take the form:
\begin{align} \label{wls}
    \min_\mb{x} \|\ol{P} \mb{x} - \mb{b}\|^2_{C^{-1}}.
\end{align}
It can be shown~\cite{du} that the saddle point equivalent of~\eqref{wls} is
\begin{align} \label{spe}
    \min_\mb{x} \max_\mb{y} \left\{ \langle \mb{y}, \mb{b} \rangle - \frac{1}{2} \| \mb{y} \|^2_C - \langle \mb{y}, \ol{P} \mb{x} \rangle \right \},
\end{align}
signifying the importance of the saddle point formulation, which enables a solution of~\eqref{wls} without evaluating the inverse of the matrix~$C$, thus decreasing the computational complexity. When the local data is distributed, i.e.,
\[
{\ol{P} := \frac{1}{n} \sum_{i=1}^n P_i},\quad
{h_i(\mb{y}):= \langle \mb{y}, \mb{b}_i \rangle - \frac{1}{2} \| \mb{y} \|^2_{C_i}},
\]
and~${H(\mb{y}) := \frac{1}{n} \sum_{i=1}^n h_i(\mb{y})}$, the above optimization problem takes the form of Problem~$\P$. Furthermore, we note that the gradients of~\eqref{spe}, with respect to~$\mb{x}$ and~$\mb{y}$, can be evaluated more efficiently as compared to~\eqref{wls}.

In several cases~\cite{du_quad,WW_GDA_mingyi}, reinforcement leaning takes the same form as weighted linear regression. The main objective in reinforcement learning is policy evaluation that requires learning the value function~$V^{\bds{\pi}}$, for any given joint policy~$\bds{\pi}$. The data~${\{s_{k}, s_{k+1}, r_{k}\}_{k=1}^N}$ is generated by the policy~$\bds{\pi}$, where~$s_{k}$ is the state and~$r_k$ is the reward at the~$k$-th time step. With the help of a feature function~$\phi(\cdot)$, which maps each state to a feature vector, we would like to estimate the model parameters~$\mb{x}$ such that~${V^{\bds{\pi}} \approx \langle \phi(s), \mb{x} \rangle}$. A well known method for policy evaluation is to minimize the empirical mean squared projected Bellman error, which is essentially weighted linear regression
% \vspace{-0.1cm}
\begin{align*}
    \min_\mb{x} \|\ol{P} \mb{x} - \mb{b}\|^2_{C^{-1}},
\end{align*}
% \vspace{-0.1cm}
where
\[
{\ol{P}:= \sum_{k=1}^N \langle \phi(s_k), \phi(s_k) - \gamma \phi(s_{k+1})\rangle},
\]
% \vspace{-0.1cm}
for some discount factor~${\gamma \in (0,1)}$,~${C:=\sum_{k=1}^N \|\phi(s_k)\|^2 }$, and~${\mb{b} := \sum_{k=1}^N r_k \phi(s_k) }$. %This can be solved by evaluating the saddle point equivalent form described in~\eqref{spe}.

\textbf{Supervised learning.} Classical supervised learning problems are essentially empirical risk minimizations. The aim is to learn a linear predictor~$\mb{x}$ when~$H(\cdot)$ is the loss function to be minimized using data matrix~$\ol{P}$, and some regularizer~$G(\cdot)$. The problem can be expressed as below:
\begin{align*}
    \min_{\mb{x}} \left\{ H(\ol{P} \mb{x}) + G(\mb{x}) \right \}
\end{align*}
\vspace{-0.1cm}
The saddle point formulation of above problem can be expressed as:~${\min_{\mb{x}} \max_{\mb{y}} \left\{ G(\mb x) + \langle \mb y, \ol{P} \mb x \rangle - H^*(\mb y)\right \}}$. For large-scale systems, the data~$P_i$ is geographically distributed among different computational nodes and the local functions~$g_i$'s and~$h_i$'s are also private. Problem~$\P$ can be obtained here by choosing~${\ol{P}:= \frac{1}{n} \sum_{i=1}^n P_i, G(\mb{x}) := \sum_{i=1}^n g_i(\mb{x})}$, and~ ${H^*(\mb{y}):=\sum_{i=1}^n h_i^*(\mb{y})}$.
\vspace{-0.2cm}
\subsection{Algorithm development and description}
In order to motivate the proposed algorithm, we first describe the canonical distributed minimization problem: ${\min_{\mb x}G(\mb x):=\tfrac{1}{n}\sum_{i=1}^n g_i(\mb x)}$, where~$G$ is a smooth and strongly convex function. A well-known distributed solution is given by~\cite{DGD_nedich,AB_proc}:
\begin{align}\label{dgd_eq}
\mb x_i^{k+1} = \sum_{r=1}^nw_{ir}
(\mb x_r^k-\alpha\cdot\mb \nabla g_i(\mb x_i^k)),
\end{align}
\vspace{-0.1cm}
where~$\mb x_i^k$ is the estimate of the unique minimizer (denoted as~$\mb x^*$ such that~${\nabla G(\mb x^*)=\tfrac{1}{n} \sum_i \nabla g_i(\mb x^*)=\mb{0}_{p_x}}$) of~$G$ at node~$i$ and time~$k$, and~$w_{ir}$ are the network weights such that~${w_{i,r}\neq 0}$, if and only if~${(i,r)\in\mc E}$, and~${W=\{w_{ir}\}}$ is primitive and doubly stochastic. Consider for the sake of argument that each node at time~$k$ possesses the minimizer~$\mb x^*$; it can be easily verified that~${\mb x_i^{k+1}\neq\mb x^*}$, because the local gradients are not zero at the minimizer, i.e.,~${\nabla g_i(\mb x^*)\neq\mb{0}_{p_x}}$. To address this shortcoming of~\eqref{dgd_eq}, recent work~\cite{GT_CDC,NEXT,DSGD_NIPS,D2,harnessing,AB} uses a certain gradient tracking technique that updates an auxiliary variable~$\mb y_i^k$ over the network such that~${\mb y_i^k\ra\tfrac{1}{n}\sum_i\nabla g_i(\mb x_k^i)}$. The resulting algorithm:
\begin{align}
\mb x_i^{k+1} &= \sum_{r=1}^nw_{ir}
(\mb x_r^k-\alpha\cdot\mb y_i^k),\\
\mb y_i^{k+1} &= \sum_{r=1}^nw_{ir} (\mb y_i^k + \mb \nabla g_i(\mb x_i^{k+1}) - \mb \nabla g_i(\mb x_i^k)),
\end{align}
\vspace{-0.1cm}
converges linearly to~$\mb x^*$ thus removing the bias caused by the local~$\nabla g_i$ vs. global gradient~$\nabla G$ dissimilarity. 

The proposed method~$\GDA$, formally described in Algorithm~\ref{algo}, uses gradient tracking in both the descend and ascend updates. In particular, there are three main components of the~$\GDA$ method: (i) gradient descent for~$\mb{x}$ updates; (ii) gradient ascent for~$\mb{y}$ updates; and (iii) gradient tracking. However, since the coupling matrices~$P_i$'s are not identical at the nodes, we add an intermediate step to implement consensus on~$P_i$'s. Initially,~$\GDA$ requires random state vectors~$\mb{x}_i^0$ and~$\mb{y}_i^0$ at each node~$i$, gradients evaluated with respect to~$\mb{x}$ and~$\mb{y}$ and some positive stepsizes~$\alpha$ and~$\beta$ for descent and ascent updates, respectively. At each iteration~$k$, every node computes gradient descent ascent type updates. The state vectors~$\mb{x}_i^{k+1}$ (and~$\mb{y}_i^{k+1}$) are evaluated by taking a step in the negative (positive) direction of the gradient of global problem, and then sharing them with the neighbouring nodes according to the network topology. It is important to note that~$\mb{q}_i^{k}$ and~$\mb{w}_i^{k}$ are the global gradient tracking vectors, i.e.,~${\mb{q}_i^{k} \ra \nabla_{\mb{x}} F(\mb{x}, \mb{y})}$ and~${\mb{w}_i^{k} \ra \nabla_{\mb{y}} F(\mb{x}, \mb{y})}$. 

\begin{algorithm}[h] \label{algo}
	\caption{~\GDA~at each node~$i$} \label{algo}
	\setstretch{1.35}
% 	\begin{spacing}{1.1}
    \begin{algorithmic}[1] 
		\Require~${\mb{x}_i^0 \in\mbb R^{p_x}, \mb{y}_i^0\in\mbb R^{p_y}}, {{P}_{i}^{0} = {P}_{i}}, {\{w_{ir}\}_{r=1}^n}, {\alpha>0,}$
		
		~~\:${\beta>0},{\mb{q}_i^0= \nabla_x f_{i}(\mb{x}_{i}^{0}, \mb{y}_{i}^{0})}, {\mb{w}_i^0=\nabla_y f_{i}(\mb{x}_{i}^{0}, \mb{y}_{i}^{0})}$
		\For{$k = 0,1,2,\dots,$},
		\State $ {P}_{i}^{k+1} \!\la \sum_{r=1}^n w_{ir} {P}_{r}^{k}$
		\State $\mb{x}_{i}^{k+1} \la \sum_{r=1}^n w_{ir} (\mb{x}_{r}^{k} - \alpha\cdot\mb{q}_{i}^{k})$
		\State $\mb{q}_{i}^{k+1} \la \sum_{r =1 }^{n}w_{ir} (\mb{q}_{r}^{k} + \nabla_x \mb{f}_i^{k+1} - \nabla_x \mb{f}_i^{k})$
% 		\Comment{Gradient Tracker update}
		\State $\mb{y}_{i}^{k+1} \la \sum_{r=1}^n w_{ir} (\mb{y}_{r}^{k} + \beta\cdot\mb{w}_{i}^{k})$ 
		\State $\mb{w}_{i}^{k+1} \!\la \sum_{r =1 }^{n}w_{ir} (\mb{w}_{r}^{k} + \nabla_y \mb{f}_i^{k+1} - \nabla_y \mb{f}_i^{k})$
% 		\Comment{Gradient Tracker update}
		\EndFor
	\end{algorithmic}
% 	\end{spacing}
\end{algorithm}

$\GDAl$: We note that~$\GDA$ implements consensus on the coupling matrices (Step~2), which can result in costly communication when the size of these matrices is large. We thus consider a special case of~$\GDA$ that does not implement consensus on the coupling matrices, namely~$\GDAl$\footnote{We do not explicitly write~$\GDAl$ as it is the same as Algorithm~\ref{algo}:~$\GDA$ but without the consensus (Step~2) on~$P_i$'s.} and characterize its convergence properties for the following three cases: 
\begin{enumerate}[(i)]
\item strongly concave-convex problems with different coupling matrices~$P_i$'s at each node; 
\item strongly concave-convex problems with identical~$P_i$'s;
\item quadratic problems with different~$P_i$'s at each node.
\end{enumerate}
\vspace{-0.5cm}
\section{Main results} \label{main_res}
Before we describe the main results, we first provide some definitions followed by the assumptions required to establish those results.
\begin{definition}[Smoothness and convexity]
A differentiable function~$G: \R^p \ra \R~$ is~$L$-smooth if~${\forall \mb{x}, \mb{y} \in \R^p,}$
\begin{align*}
    \|\nabla G(\mb{x}) - \nabla G(\mb{y}) \| \leq L \| \mb{x} - \mb{y} \|
\end{align*}
and~$\mu$-strongly convex if~${\forall \mb{x}, \mb{y} \in \R^p,}$
\begin{align*}
    G(\mb{y}) + \langle \nabla G(\mb{y}), \mb{x} - \mb{y} \rangle + \frac{\mu}{2} \| \mb{x} - \mb{y} \|^2 \leq G(\mb{x}) .
\end{align*}
It is of significance to note that if~$G(\cdot)$ is~$L$ smooth, then it is also~$(L + \xi)$ smooth,~$\forall \xi > 0$.
\end{definition}
\newpage

\begin{definition}[Conjugate of a function] \label{conj_def} The conjugate of a function~$H:\R^p \ra \R$ is defined as
\begin{align*}
    H^*(\mb{y}) := \sup_{\mb{x}\in \R^p} \left\{ \langle \mb{x}, \mb{y} \rangle - H(\mb{x})\right\}, \quad \forall \mb{y} \in \R^p.
\end{align*}
Moreover, if~$H(\cdot)$ is closed and convex, then~${[H^{*}(\cdot)]^* = H(\cdot)}$, and if~$H(\cdot)$ is~$L$-smooth and~$\mu$-strongly convex, then~$H^*(\cdot)$ is~$\tfrac{1}{\mu}$-smooth and~$\tfrac{1}{L}$-strongly convex~\cite{conj_sm_sc}. It is also of relevance to note that the gradient mappings are inverses of each other, i.e.,~$\nabla H^*(\cdot) = (\nabla H)^{-1}(\cdot)$~\cite{rockafellar}.
\end{definition}

% \begin{dfn}(Some useful constants)
% We define~$L := \max \{L_1, L_2\}$ and the condition number of~$H(\cdot)$ as~$L_2/ \mu$. Furthermore, we denote~$\kappa:= L/ \mu \geq L_2/ \mu$.
% \end{dfn}
% \ subsection*{Assumptions}
Next, we describe the assumptions under which the convergence results of~$\GDA$ will be developed; note that all of these assumptions may not be applicable at the same time.

\begin{assump}[Smoothness and convexity] \label{sm_scc}
Each local~$g_i$ is~$L_1$-smooth and each~$h_i$ is~$L_2$-smooth, where~$L_1,L_2$ are arbitrary positive constants. Furthermore, the global~$G$ is convex and the global~$H$ is~$\mu$-strongly convex.
\end{assump}

\begin{assump}[Quadratic] \label{quad}
The~$g_i$'s and~$h_i$'s are quadratic functions, i.e., 
\begin{align*}
g_i(\mb x) &:= \mb x^\top Q_i\mb x + \mb{q}_i^\top \mb x + {q}_i,\\
h_i(\mb y) &:= \mb y^\top R_i\mb y + \mb{r}_i^\top \mb y + {r}_i,
\end{align*}

\noindent such that ${\mb{q}_i \in \R^{p_x}}$, ${\mb{r}_i \in \R^{p_y}}$, ${q_i, r_i \in \R}$,~${Q_i \in \R^{p_x \mt p_x}}$, and~${R_i \in \R^{p_y \mt p_y}}$,~$\forall i$. Moreover, for~${\ol{Q} := \frac{1}{n} \sum_{i=1}^n Q_i}$ and~${\ol{R} := \frac{1}{n} \sum_{i=1}^n R_i}$, we assume that~${(\ol{Q}+\ol{Q}^\top)}$ is positive definite and~${(\ol{R}+\ol{R}^\top)}$ is positive semi-definite.
\end{assump}

\begin{assump}[Full ranked coupling matrix] \label{cp_rank}
The coupling matrix~${\ol{P}:= \tfrac{1}{n} \sum_i P_i}$ has full column rank.
\end{assump}

\begin{assump}[Doubly stochastic weights] \label{doub_stoc}
The weight matrices~${W:=\{w_{i,r}\}}$ associated with the network are primitive and doubly stochastic, i.e.,~$ W \mb{1}_n = \mb{1}_n$ and~$\mb{1}_n^\top W = \mb{1}_n^\top$. 
\end{assump}

We note that Assumption~\ref{sm_scc} does not require strong convexity of~$H$ while Assumptions~\ref{sm_scc} and~\ref{cp_rank} are necessary for linear convergence~\cite{du}. The primal problem~${\min_{\mb{x}} \theta(\mb{x})}$ requires poly~$(\epsilon^{-1})$ iterations to obtain an~$\epsilon$-optimal solution even in the centralized case if we ignore any of the above assumptions. It is important to note that Assumptions $1$-$3$ are not applicable simultaneously;~$\GDA$ and~$\GDAl$ are analyzed under different assumptions, clearly stated in each theorem. Next, we define some useful constants to explain the main results. Let~${L := \max \{L_1, L_2\}}$ and let the condition number of~$H(\cdot)$ be~$L_2/\mu$. Furthermore, we denote~${\kappa:= L/\mu \geq L_2/\mu}$. The maximum and minimum singular values of the coupling matrix~$\ol{P}$ are defined as~$\sigma_M$ and~$\sigma_m$, respectively. Moreover, the condition number for~$\ol{P}$ is denoted by~${\gamma:= \sigma_M/\sigma_m}$.
% {\uk we don't make all these assumptions simultaneously. These are not global assump. The proposed algos will be analyzed under different assumptions }
\vspace{-0.6cm}
\subsection{Convergence results for~$\GDA$}
We now provide the main results on the convergence of~$\GDA$ and discuss their attributes.

\begin{theorem} \label{th1}
Consider Problem~$\P$ under Assumptions~\ref{sm_scc},~\ref{cp_rank}, and~\ref{doub_stoc}. For a large enough positive constant~$c>0$, assume the stepsizes are such that 
\begin{align*}
\alpha = \ol{\alpha} &:= \ol{\beta} \frac{{\mu}^2}{c \sigma_M^2},\\
\beta = \ol{\beta} &:= \min \bigg \{ \frac{\sigma_m^2 (1 - \lambda)^2}{192 \sigma_M^2 L}, \frac{L (1 - \lambda)^2}{48 \sigma_M^2},
    \frac{1}{382 \kappa L} \bigg \}.
\end{align*}
Then~$\GDA$ achieves an~$\epsilon$-optimal solution in 
\begin{align*}
    \mc{O} \left( \max \bigg \{ \frac{\gamma^6 \kappa^3}{(1 - \lambda)^4}, \frac{\sigma_m^2 \sigma_M^2 \kappa}{L^2 \mu^2 (1 - \lambda)^4},
    \gamma^2 \kappa^5 \bigg \} \log \frac{1}{\epsilon} \right)
\end{align*}
gradient computations (in parallel) at each node
\end{theorem}

% \newpage
\begin{cor} \label{crr}
Consider Problem~$\P$ under Assumptions~\ref{sm_scc},~\ref{cp_rank}, and~\ref{doub_stoc}, and~$\GDA$ with stepsizes~${\alpha := \ol{\alpha}}$ and~${\beta:=\ol{\beta}}$. If
\begin{align*}
    \gamma^2 \kappa^4 \geq \max \bigg \{ \frac{\gamma^6 \kappa^2}{(1 - \lambda)^4}, \frac{\sigma_m^2 \sigma_M^2}{L^2 \mu^2 (1 - \lambda)^4}\bigg \}, 
\end{align*}
then~$\GDA$ achieves an~$\epsilon$-optimal solution linearly at a network-independent convergence rate of~$\mc{O} \left(\gamma^2 \kappa^5 \log \tfrac{1}{\epsilon}\right)$.
% If, on the other hand,~${\frac{\sigma_M}{\sigma_m} \geq (1-\lambda)^2}$
% then~$\GDA$ achieves an~$\epsilon$-optimal solution linearly at a network independent convergence rate of:
% \begin{align*}
%     \mc{O} \left(\max\left\{\kappa^5, \frac{\kappa^5\sigma_M^2}{{\mu}^2 }\right\} \log \frac{1}{\epsilon}\right).
% \end{align*}
\end{cor}

We now discuss these results in the following remarks.

\begin{rmk}[Linear convergence]
$\GDA$ eliminates the dissimilarity caused by heterogeneous data at each node using gradient tracking in both of the~$\mb{x}_i^k$ and~$\mb{y}_i^k$ updates. Theorem~\ref{th1} provides an explicit  linear rate at which~$\GDA$ converges to the unique saddle point~$(\mb{x}^*, \mb{y}^*)$ of Problem~$\P$.
\end{rmk}
% \begin{rmk}[Linear Convergence]
% We note that, at each node~$i$,~$\GDA$ evaluates~$\mb{x}_i^k$ and~$\mb{y}_i^k$ updates in parallel. This significantly increases the computational speed (as compared to the centralized counter-part) for practical implementation on systems, which are capable of parallel processing. We show empirical evidence of linear speedup in Section~\ref{sims}.
% \end{rmk}

\begin{rmk}[Network-independence]The convergence rates for distributed methods are always effected by network parameters, like connectivity or the spectral gap:~$(1-\lambda)$. Corollary~\ref{crr} explicitly describes a regime in which the convergence rate of~$\GDA$ is independent of the network topology. 
% We note that this regime can be attained if the largest singular value of the coupling matrix~${\sigma_M \geq \frac{\mu}{(1-\lambda)^2}}$. Moreover, for a well connected network,~$\lambda \approx 0$ and~$\sigma_M \geq \mu$.
\end{rmk}

\begin{rmk}[Communication complexity]At each node, $\GDA$ communicates two~$p_x$-dimensional vectors, two~$p_y$-dimensional vectors, and a~${p_x \mt p_y}$ dimensional coupling matrix, per iteration. In ad hoc peer-to-peer networks, the node deployment may not be deterministic. Let~$\omega$ be the expected degree of the underlying (possibly random) strongly connected communication graph. Then the expected communication complexity required for~$\GDA$ to achieve an~$\epsilon$-optimal solution is
\begin{align*}
    \mc{O} \left( \omega p_xp_y\max \bigg \{ \frac{\gamma^6 \kappa^3}{(1 - \lambda)^4}, \frac{\sigma_m^2 \sigma_M^2 \kappa}{L^2 \mu^2 (1 - \lambda)^4},
    \gamma^2 \kappa^5 \bigg \}  \log \frac{1}{\epsilon} \right)
\end{align*}
scalars per node. We note that~$\omega$ is a function of underlying graph, e.g.,~${\omega=\mc{O}(1)}$ for random geometric graphs and~${\omega=\mc{O}(\log n)}$ for random exponential graphs.
% \begin{align*}
%     \mc{O} \left(\frac{\zeta \kappa^5\sigma_M^2}{\sigma_m^2} \max \left\{\frac{1}{(1-\lambda)^4}, \frac{\sigma_M^2}{{\mu}^2} \right\} \log \frac{1}{\epsilon}\right).
% \end{align*}
\end{rmk}

% \begin{rmk}[Comparison with the centralized]

% \end{rmk}

\subsection{Convergence results for~$\GDAl$}
We now discuss~$\GDAl$ in the context of the aforementioned special cases below. 

\begin{figure*}[!h]
\centering
\subfigure{\includegraphics[width=2.2in]{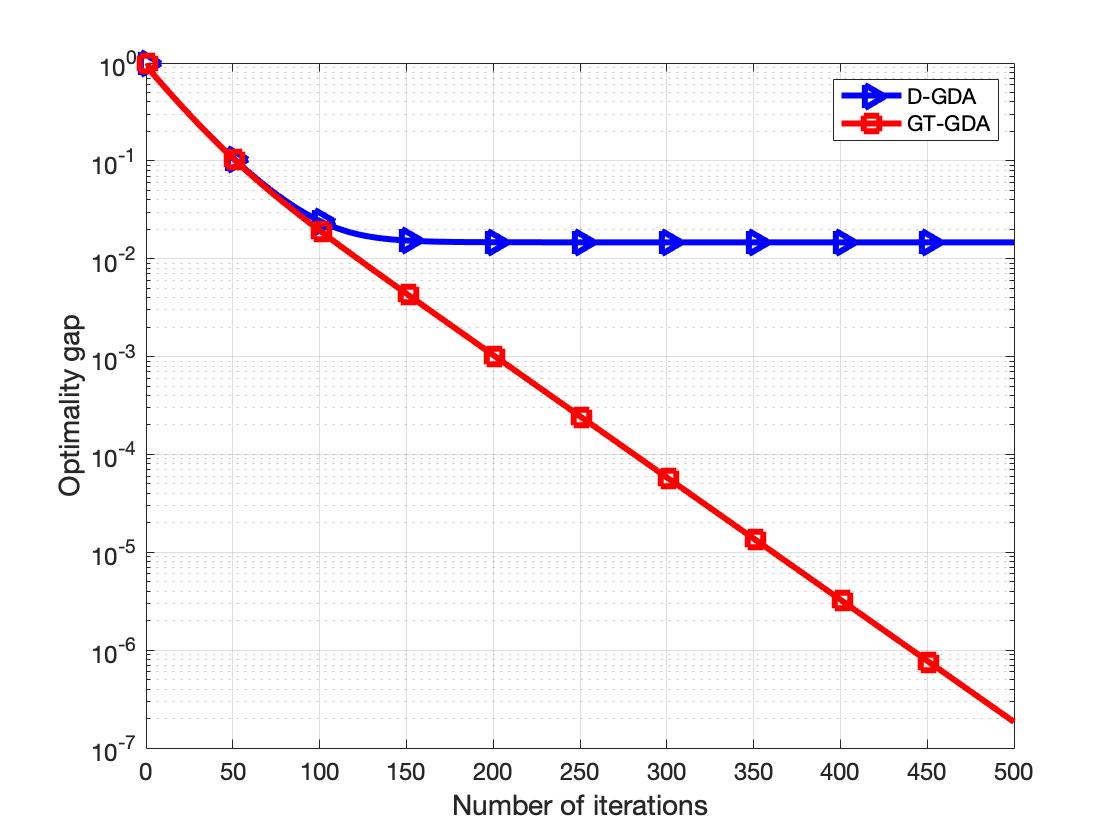}}
\hspace{0.1cm}
\subfigure{\includegraphics[width=2.2in]{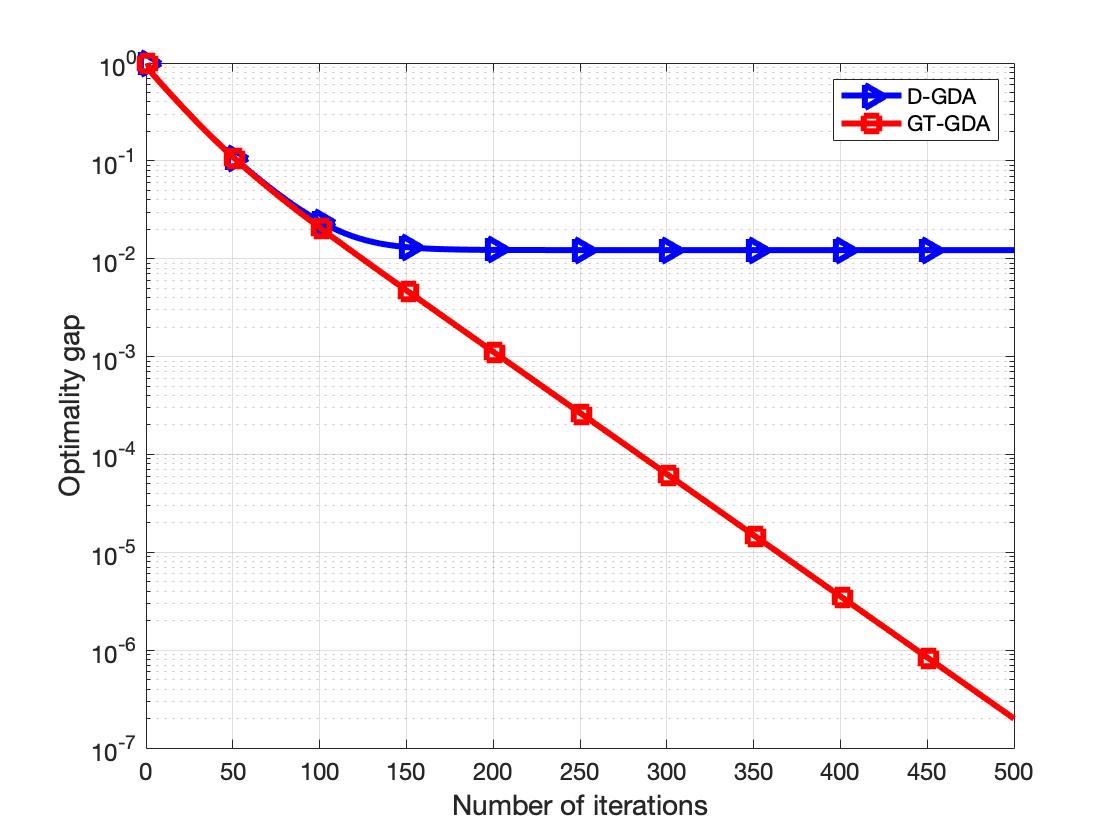}}
\hspace{0.1cm}
\subfigure{\includegraphics[width=2.2in]{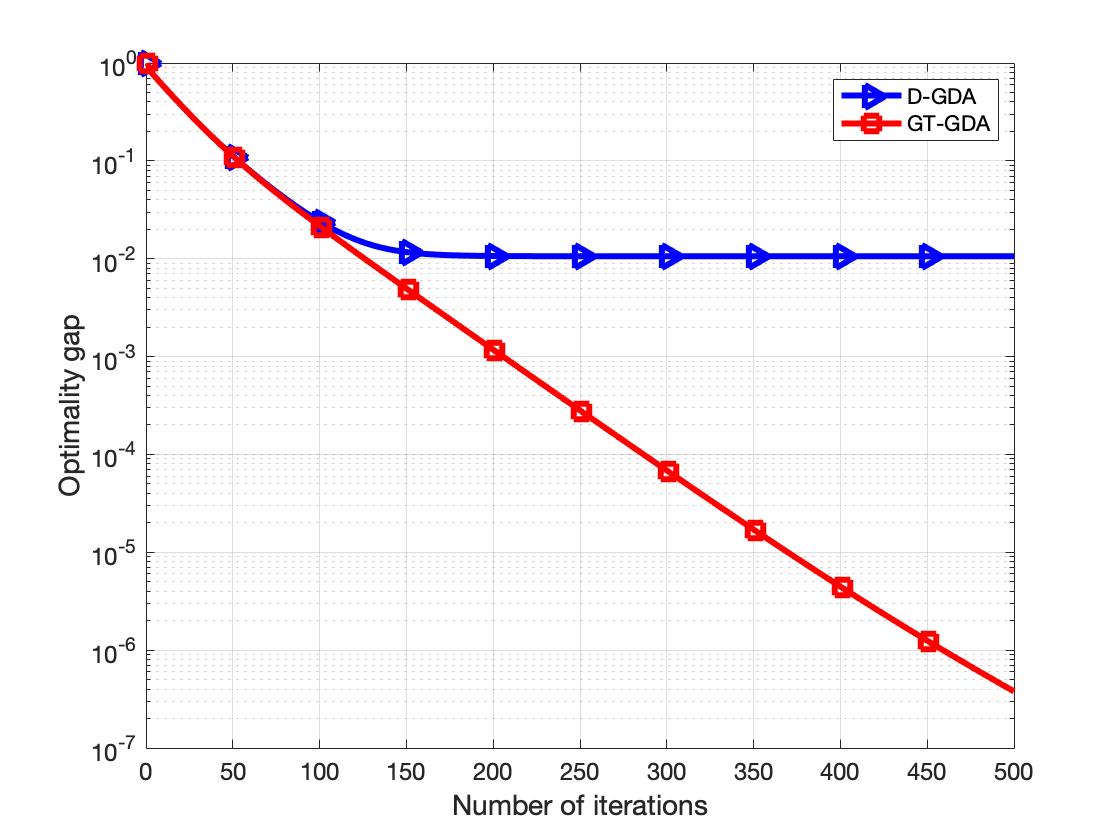}}
\caption{Performance comparison of~$\DGDA$ and~$\GDA$ over a network of~${n=8}$ nodes (left),~${n=32}$ nodes (center) and~${n=100}$ nodes (right).}
\label{r1}
\end{figure*}

\begin{theorem}[$\GDAl$ for Problem~$\P$] \label{th2}
Consider Problem~$\P$ under Assumptions~\ref{sm_scc},~\ref{cp_rank}, and~\ref{doub_stoc}. If the stepsizes~${\alpha \in (0, \ol{\alpha}]}$ and~${\beta \in (0, \ol{\beta}]}$, then~$\GDAl$ converges linearly to an error ball around the unique saddle point.
% that is controlled by the stepsizes~$\alpha$ and~$\beta$. 
\end{theorem}

\begin{rmk}[Convergence to an inexact solution] 
We note that the speed of convergence for~$\GDAl$ is of the same order as~$\GDA$, however,~$\GDAl$ converges to an error ball around the unique saddle point, which depends on the size of~${\tau := \tfrac{1}{n}\mn{ \sum_{i=1}^n (P_i - \ol{P})}}$. This error~$\tau$ can be eliminated by using identical~$P_i$'s at each node or by having consensus. 
% or by using decaying stepsizes~${\alpha_k:=\mc{O}(1/k)}$ and~${\beta_k:=\mc{O}(1/k)}$. 
The first possibility is considered in the next theorem and the second is explored in~$\GDA$. 
% The last case would result in slower (sub-linear) convergence.
\end{rmk}

\begin{theorem}[$\GDAl$ for Problem~$\P$ with same~$P_i$'s] \label{th3}
Consider Problem~$\P$ under Assumptions~\ref{sm_scc},~\ref{cp_rank}, and~\ref{doub_stoc} and with identical~$P_i$'s at each node. If the stepsizes~${\alpha = \ol{\alpha}}$ and~${\beta = \ol{\beta}}$, then the computational complexity of~$\GDAl$ to achieve~$\epsilon$-optimal solution is the same as in Theorem~1, whereas the communication cost reduces by a factor of~$\mc O(\min(p_x,p_y))$.
% where the communication complexity is given by {\uk degree??? see SPM paper}
% \begin{align*}
%     \mc{O} \left( \max \bigg \{ \frac{\gamma^6 \kappa^3}{(1 - \lambda)^4}, \frac{\sigma_m^2 \sigma_M^2 \kappa}{L^2 \mu^2 (1 - \lambda)^4},
%     \gamma^2 \kappa^5 \bigg \} \right)
% \end{align*}
% \begin{align*}
%     \mc{O} \left(\frac{\omega \kappa^5\sigma_M^2}{\sigma_m^2} \max \left\{\frac{1}{(1-\lambda)^4}, \frac{\sigma_M^2}{{\mu}^2} \right\} \log \frac{1}{\epsilon}\right)
% \end{align*}
% where~${\omega := \max \{p_x, p_y\}}$.
\end{theorem}

\begin{rmk}[Reduced communication complexity]We note for~$\GDAl$, each node communicates two~$p_x$ dimensional vectors and two~$p_y$ dimensional vectors per iteration. For large values of~${\max\{p_x, p_y\}}$, this is significantly less than what is required for~$\GDA$, i.e.,~$\mc{O}(p_x p_y)$. This makes~$\GDAl$ more convenient for applications where communication budget is low.
\end{rmk}

\begin{theorem}[$\GDAl$ for quadratic problems] \label{th4}
Consider Problem~$\P$ under Assumptions~\ref{quad},~\ref{cp_rank}, and~\ref{doub_stoc} (with different $P_i$'s at the nodes). If the stepsizes~$\alpha$ and~$\beta$ are small enough, then~$\GDAl$ converges linearly to the unique saddle point~${(\mb{x}^*, \mb{y}^*)}$ without consensus on~$P_i$'s.
\end{theorem}

\begin{rmk}[Exact analysis]
The convergence analysis we provide for the quadratic case is exact. In other words, we do not use the typical norm bounds and derive the error system of equations as an exact LTI system. Using the concepts from matrix perturbation theory for semi-simple eigenvalues, we show that~$\GDAl$ linearly converges to the unique saddle point of~$\P$ with quadratic cost functions.
\end{rmk}

\section{Simulations} \label{sims}
We now provide numerical experiments to compare the performance of distributed gradient descent ascent with ($\GDA$) and without gradient tracking ($\DGDA$) and verify the theoretical results. We would like to perform a preliminary empirical evaluation on a linear regression problem. We consider the problem of the form:
\begin{align} \label{pb}
    \min_{\mb{x}} \frac{1}{2n} \|\ol{P} \mb{x} - \mb{b}\|^2 + \lambda R(\mb{x}); 
\end{align}
and the saddle point equivalent of above problem is
\begin{align*}
    \min_\mb{x} \max_\mb{y} \left\{ \langle \mb{y}, \mb{b} \rangle - \frac{1}{2} \| \mb{y} \|^2 - \langle \mb{y}, \ol{P} \mb{x} \rangle  + \lambda R(\mb{x}) \right \}.
\end{align*}

Performance characterization using the saddle point form of~\eqref{pb} is common in the literature available on centralized gradient descent ascent~\cite{du, mokhtari}. For large-scale problems when data is available over a geographically distributed nodes, decentralized implementation is often preferred. In this paper, we consider the network of nodes communicating over strongly connected networks of different sizes and connectivity to extensively evaluate the performance of~$\GDA$. Figure~\ref{net} shows two directed exponential networks of~${n=8}$ and~${n=32}$ nodes. We note that although they are directed, their corresponding matrices~$W$ are weight-balanced. To highlight the significance of distributed processing for large-scale problems, we evaluate the simulation results with the networks shown in Fig.~\ref{net} and their extensions to~${n=100}$ and~${n=200}$ nodes.
\begin{figure}[!ht]
\centering
\subfigure{\includegraphics[width=1.7in]{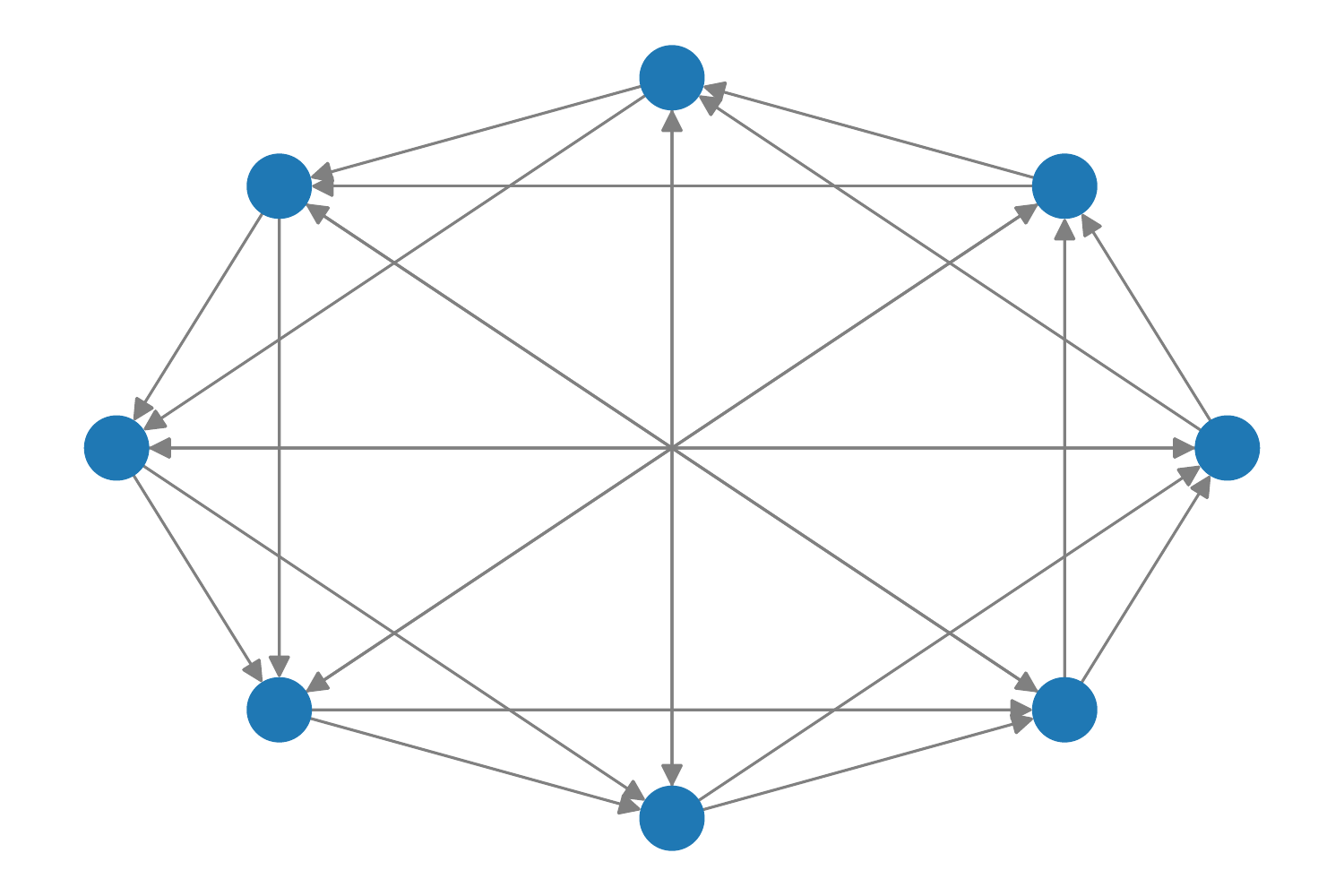}}
\subfigure{\includegraphics[width=1.7in]{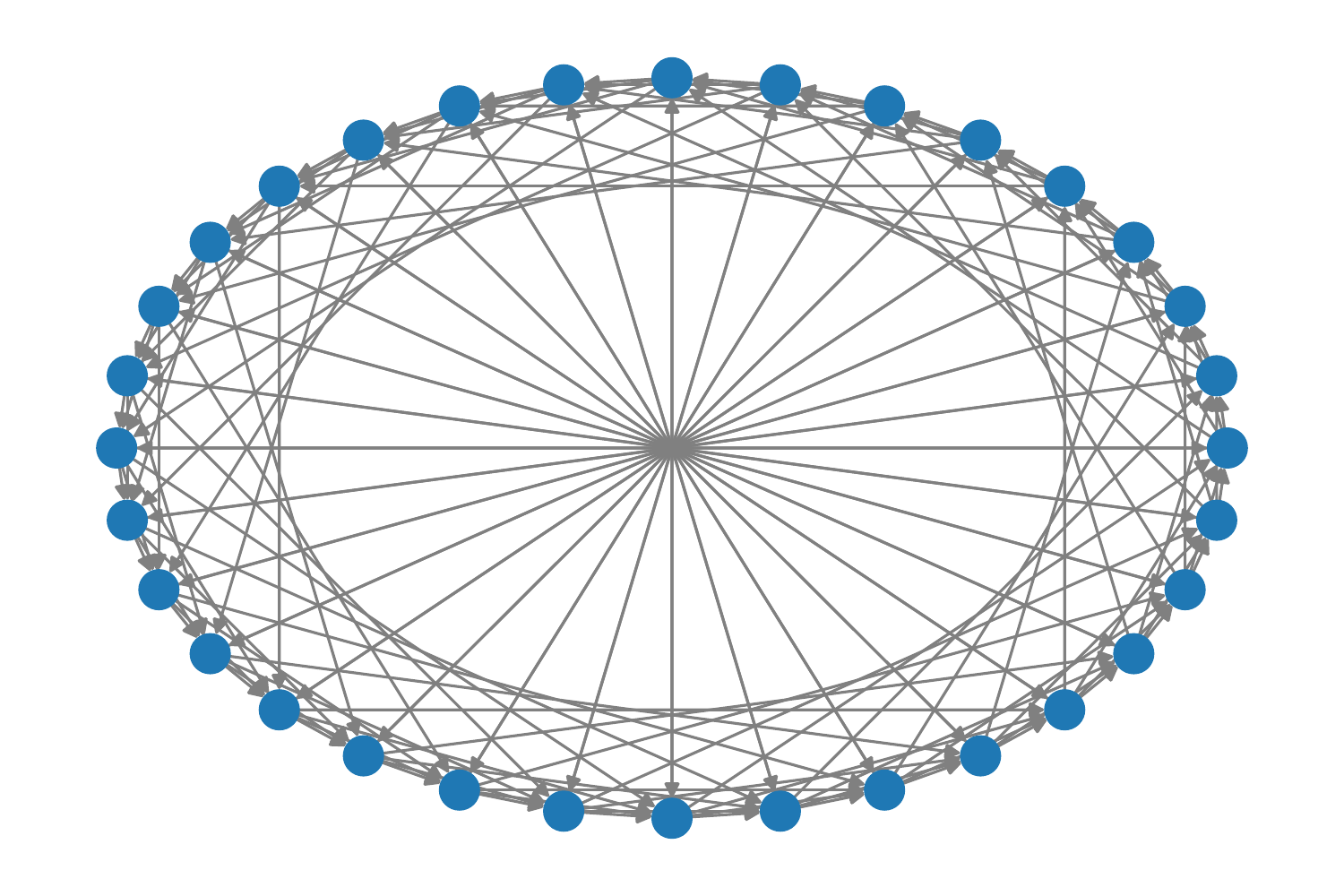}}
\caption{Directed exponential graphs with~${n=8}$ nodes (left) and~${n=32}$ nodes (right).}
\label{net}
\end{figure}

\textbf{Smooth and strongly convex regularizer:} We first consider~\eqref{pb} with smooth and strongly convex regularizer~${R(\mb{x}):= \|\mb{x}\|^2_C}$. Therefore, the resulting problem is strongly-convex strongly-concave. For a peer-to-peer mesh network of~$n$ nodes, each node~$i$ has its private~${\mb{b}_i \in \R^{p_x}}$ and~${C_i \in \R^{p_y \mt p_x}}$ such that the average~${\mb{b} := \frac{1}{n} \sum_{i=1}^n \mb{b}_i}$ and~${C := \frac{1}{n} \sum_{i=1}^n C_i}$, and~${\ol{P}:= \frac{1}{n} \sum_{i=1}^n P_i}$ has full column rank. We set the dimensions~$p_x = 10$ and~$p_y = 4$ and evaluate the performance of~$\GDA$ for data generated by a random Gaussian distribution. 

We characterize the performance by evaluating the optimality gaps:~${\|\ol{\mb{x}}^k - \mb{x}^*\|}+{\|\ol{\mb{y}}^k - \mb{y}^*\|}$. Fig.~\ref{r1} represents the comparison of the simulation results of~$\DGDA$ and~$\GDA$ for different sizes of exponential networks (${n=8,32}$ and~$100$); some shown in figure~\ref{net}. The optimality gap reduces with the increase in the number of iterations. It can be observed that~$\DGDA$ (blue curve) converges to an inexact solution because it evaluates gradients with respect to it's local data at each step; hence move towards local optimal. On the contrary, the proposed method~$\GDA$ (red curve) uses gradient tracking and consistently converge to the unique saddle point of the global problem.
% Consider~\eqref{sim} with smooth but non-strongly convex regularizer~$R(\mb{x}):= \sum_{i=1}$
% \begin{figure*}[!h]
% \centering
% \subfigure{\includegraphics[width=2in]{Figs/f8}}
% \hspace{0.3cm}
% \subfigure{\includegraphics[width=2in]{Figs/f32}}
% \hspace{0.3cm}
% \subfigure{\includegraphics[width=2in]{Figs/f100}}
% \caption{Performance comparison of~$\DGDA$ and~$\GDA$ over a network of~${n=8}$ nodes (left),~${n=32}$ nodes (center) and~${n=100}$ nodes (right).}
% \label{r1}
% \end{figure*}
% \newpage

\textbf{Smooth and convex regularizer:} Next we use a smooth but \textit{non} strongly convex regularizer~\cite{Schmidt}:
\begin{align*}
    R(\mb{x}):=\sum_{i=1}^{n} \sum_{j=1}^{p_x} \left[\frac{1}{t_i} \left\{\log(1+e^{t_i x_j}) + \log(1+e^{-t_i x_j}) \right\} \right].
\end{align*} 
Figure~\ref{f2} shows the results for~$\GDA$ over a network of~${n=32}$ and~${n=200}$ nodes. It can be seen that~$\GDA$ converges linearly to the unique saddle point, as it's optimality gap decreases, meanwhile~$\DGDA$ exhibits a similar convergence rate but settles for an inexact solution due to heterogeneous nature of data at different nodes. We note that~$\DGDA$ and~$\GDA$ converges to the same solution when the local cost functions are homogeneous, i.e., same data at each node. Such cases are still of significance for many applications where we have resources for parallel processing to boost computational speed.
\begin{figure}[!h]
\centering
\subfigure{\includegraphics[width=1.6in]{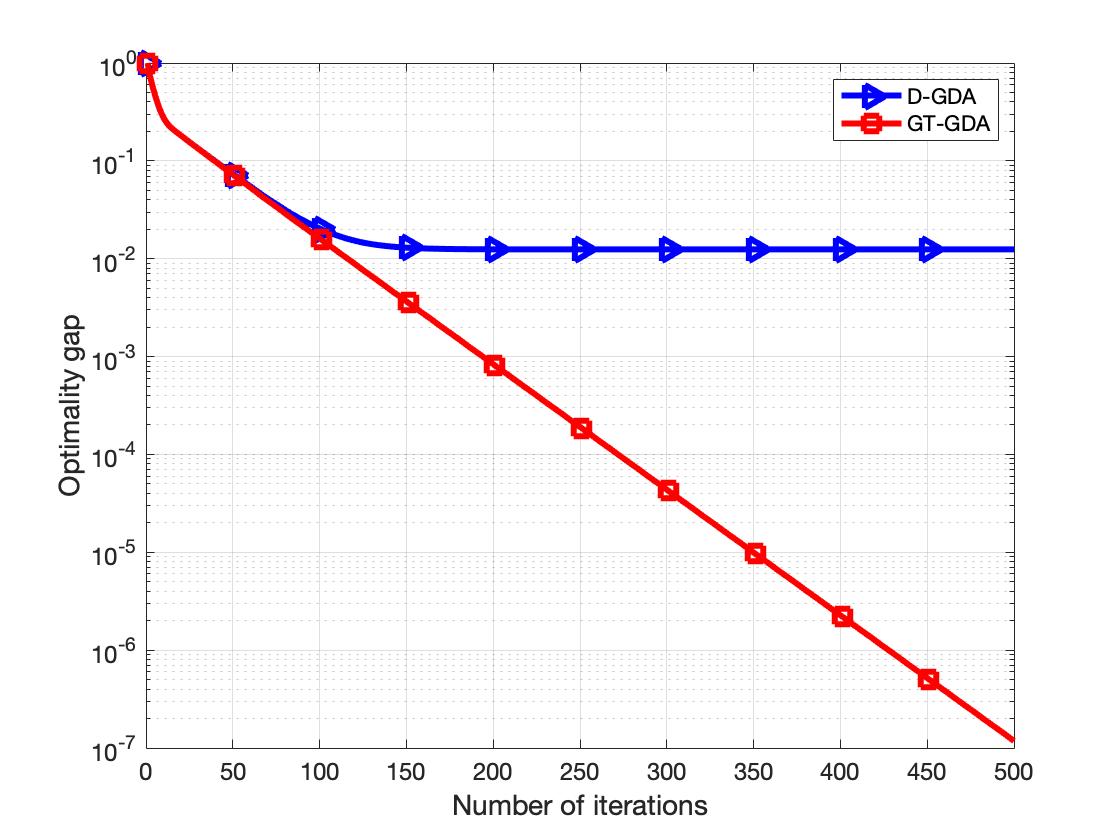}}
% \hspace{0.3cm}
\subfigure{\includegraphics[width=1.6in]{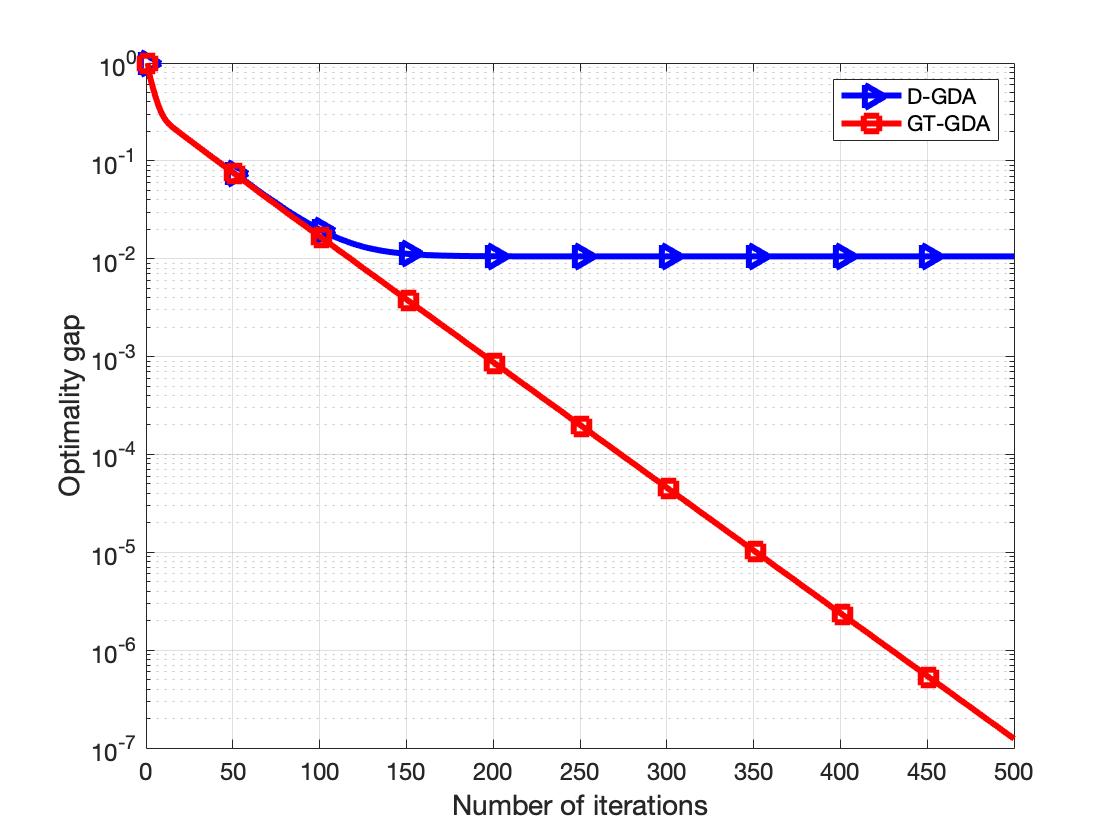}}
\caption{Performance comparison of~$\DGDA$ and~$\GDA$ over a network of~${n=32}$ nodes (left) and~${n=200}$ nodes (right).}
\label{f2}
\end{figure}

\textbf{Linear speedup:} Finally, we illustrate linear speed-up of~$\GDA$ as compared to it's centralized counterpart. We plot the ratio of the number of iterations required to attain an optimality gap of~$10^{-14}$ for~$\GDA$ as compared to the centralized gradient descent ascent method in Fig.~\ref{f3}. The results demonstrate that the performance improves linearly as the number of node increases (${n=8, 16, 32, 100, 200}$). We note that for~$n$ nodes, the centralized case has~$n$ times more data to work with at each iteration and thus has a slower convergence. In distributed setting, the processing is done in parallel which results in a faster overall performance.
\begin{figure}[!h]
\centering
\includegraphics[width=2in]{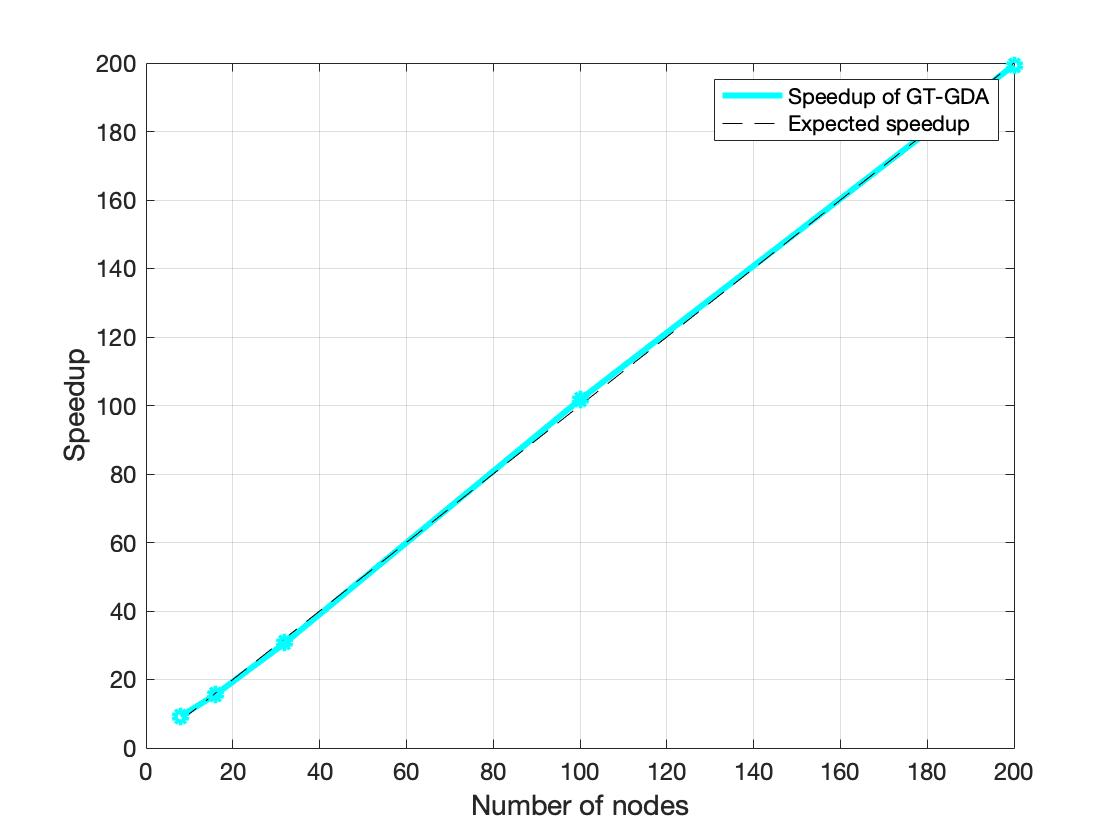}
% \hspace{0.3cm}
\caption{Linear speedup: Performance ratio of~$\GDA$ with it's centralized counterpart to achieve optimality gap of~$10^{-14}$.}
\label{f3}
\end{figure}

% \newpage
\section{Convergence analysis}\label{conv_ana}
In this section, our aim is to establish linear convergence of the proposed algorithms to the unique saddle point (under given set of assumptions) for problem class $\P$. To this aim, we define the following error quantities with the goal of characterizing their time evolution in order to establish that the error decays to zero:
\begin{enumerate}[(i)]
    \item Agreement errors,~${\| \mb{x}^k - W_1^\infty \mb{x}^k \|}$ and~${\| \mb{y}^k - W_2^\infty \mb{y}^k \|}$: Note that we define~${W^\infty := \lim_{k \ra \infty} W^k = \tfrac{1}{n} \mb 1_n\mb 1_n^\top}$, ${W_1:=W \otimes I_{p_x}}$, ${W_2:=W \otimes I_{p_y}}$ (where~${\otimes}$ denotes the Kronecker product), and thus each error quantifies how far the network is from agreement;
    
    \item Optimality gaps,~${\| \ol{\mb{x}}^k - \mb{x}^* \|}$ and~${\| \ol{\mb{y}}^k - \mb{y}^* \|}$ or ${\| \ol{\mb{y}}^k - \nabla H^*(\ol{P} \ol{\mb{x}}^k) \|}$: Note that~${\ol{\mb{x}}^{k}:= \frac{1}{n} \sum_{i=1}^n {\mb{x}}_i^{k}}$, ${\ol{\mb{y}}^{k}:= \frac{1}{n} \sum_{i=1}^n {\mb{y}}_i^{k}}$, and each error quantifies the discrepancy between the network average and the unique saddle point $(\mb{x}^*, \mb{y}^*)$;
    
    \item Gradient tracking errors,~${\| \mb{q}^k - W_1^\infty \mb{q}^k \|^2}$ and ${\| \mb{w}^k - W_2^\infty \mb{w}^k \|^2}$: Note that these errors quantify the difference between the local and global gradients.
\end{enumerate} 
% \begin{enumerate}[(i)]
%     \item Network agreement error for $\mb{x}^k$: ${\|\mb{x}^{k} - W^\infty_1 \mb{x}^{k}\|} $;
%     \item Network agreement error for $\mb{y}^k$: $\|\mb{y}^{k} - W^\infty_2 \mb{y}^{k}\|$;
%     \item Optimality gap for $\mb{x}^k$:~${\|\ol{\mb{x}}^{k} - \mb{x}^*\|}$, \item Optimality gap for $\mb{y}^k$:~${\|\ol{\mb{y}}^{k} - \nabla H^* (\ol{P} \ol{\mb{x}}^k)\|}$;
%     \item Gradient tracking error for $\mb{x}^k$:~${\|\mb{q}^{k} - W^\infty_1 \mb{q}^{k}\|}$;
%     \item Gradient tracking error for $\mb{y}^k$:~${\|\mb{w}^{k} - W^\infty_2 \mb{w}^{k}\|}$;
% \end{enumerate}
% where $\ol{\mb{x}}^{k}:= \frac{1}{n} \sum_{i=1}^n {\mb{x}}_i^{k}$ and $\ol{\mb{y}}^{k}:= \frac{1}{n} \sum_{i=1}^n {\mb{y}}_i^{k}$.
% We note that the network agreement error quantifies the discrepancy between local $\mb{x}_i$ and average $\ol{\mb{x}}$. The optimality gap evaluates the error due to the difference between $\ol{\mb{x}}$ and the optimal $\mb{x}^*$. The gradient tracking error ensures that the direction of each step is towards the global optimal. 
\subsection{Convergence of~$\GDA$}
The following lemma provides a relationship between the error quantities defined above with the help of an LTI system describing~$\GDA$.

\begin{lem} \label{lem1}
Consider~$\GDA$ described in Algorithm~\ref{algo} under Assumptions~\ref{sm_scc},~\ref{cp_rank}, and~\ref{doub_stoc}. We define~${\mb{u}^{k}, \mb{s}^{k} \in \R^6}$ as
\begin{align*}
\mb{u}^{k} \!&:=\!  \left[ \begin{array}{c}
    \|\mb{x}^{k} - W_1^\infty \mb{x}^{k}\| \\
    \sqrt{n} \|\ol{\mb{x}}^{k} - \mb{x}^*\| \\
    L^{-1} \|\mb{q}^{k} - W_1^\infty \mb{q}^{k}\| \\
    \|\mb{y}^{k} - W_2^\infty \mb{y}^{k}\| \\
    \sqrt{n} \|\ol{\mb{y}}^{k} - \nabla H^* (\ol{P} \ol{\mb{x}}^k)\| \\
    {L}^{-1} \|\mb{w}^{k} - W_2^\infty \mb{w}^{k}\|
  \end{array}\right], \qquad 
  \mb{s}^{k} :=  \left[ \begin{array}{c}
    \|\mb{x}^{k}\| \\
    \|\mb{y}^{k}\| \\
    0 \\
    0 \\
    0 \\
    0 
  \end{array}\right],
\end{align*}
,and let ${N_{\alpha, \beta, k} \in \R^{6 \mt 6}}$ be such that it has~$\alpha \lambda^k \tau$ and~$\beta \lambda^k \tau$ at the~$(2,1)$ and~$(1,5)$ locations, respectively, and zeros everywhere else. We note that~${\tau:=\mn{ P^{0} - W_2^\infty P^{0}}}$ where~$P^0$ concatenates~$P_i$'s initially available at each node. For all~${k \geq 0}$,~${\alpha, \beta > 0}$, and~${\alpha \leq \beta \frac{{\mu}^2}{c \sigma_M^2}}$, we have
\begin{align} \label{sys_eq}
    \mb{u}^{k+1} \leq M_{\alpha,\beta} \mb{u}^{k} +  N_{\alpha,\beta,k} \mb{s}^{k},
\end{align}
where the system matrix~$M_{\alpha,\beta}$ is defined in Appendix~\ref{appA}.
\end{lem}

The proof of above lemma can be found in Appendix~\ref{p_lem1}. The main idea behind the proof is to establish bounds on error terms (elements of~${\mb{u}^k}$) for descent in~${\mb{x}}$ and ascent in~${\mb{y}}$ for a range of stepsizes~${\alpha}$ and~${\beta}$. It is noteworthy that the coupling between the ascent and descent equations gives rise to additional terms (see~${M_{\alpha, \beta}}$ in Appendix~\ref{appA}) adding to the complexity of analysis. Moreover, the analysis requires a careful manipulation of the two stepsizes unlike the existing approaches. With the help of this lemma, our goal is to establish convergence of~$\GDA$ and further characterize its convergence rate. To this aim, we first show that the spectral radius~${\rho (M_{\alpha, \beta})}$ of the system matrix is less than~$1$ in the following lemma.
\begin{lem} \label{lem2}
Consider~$\GDA$ under Assumptions~\ref{sm_scc},~\ref{cp_rank}, and~\ref{doub_stoc}. For a large enough positive constant~${c>0}$, assume the stepsizes are~${\alpha \in \left(0, \ol{\beta} \tfrac{{\mu}^2}{c \sigma_M^2} \right]}$ and~${\beta \in (0, \ol{\beta}]}$, such that 
\begin{align*}
\ol{\beta} &:= \min \left \{ \frac{\sigma_m^2 (1 - \lambda)^2}{192 \sigma_M^2 L}, \frac{L (1 - \lambda)^2}{48 \sigma_M^2},
    \frac{1}{382 \kappa L} \right \},
\end{align*}
then~${\rho (M_{\alpha, \beta}) \leq \eta < 1}$, where
\begin{align*}
    \eta &:= 1 - \mc{O} \left( \min \left \{ \frac{(1 - \lambda)^4}{c \gamma^6 \kappa^3}, \frac{L^2 \mu^2 (1 - \lambda)^4}{c \sigma_m^2 \sigma_M^2 \kappa},
    \frac{1}{c \gamma^2 \kappa^5} \right \} \right).
\end{align*}
% Then~$\GDA$ achieves an~$\epsilon$-optimal solution in 
% \begin{align*}
%     \mc{O} \left( \max \bigg \{ \frac{\gamma^6 \kappa^3}{(1 - \lambda)^4}, \frac{\sigma_m^2 \sigma_M^2 \kappa}{L^2 \mu^2 (1 - \lambda)^4},
%     \gamma^2 \kappa^5 \bigg \} \log \frac{1}{\epsilon} \right)
% \end{align*}
% gradient computations (in parallel) at each node
\end{lem}

\begin{proof}
Recall that~$M_{\alpha, \beta}$ is a non-negative matrix. From~\cite{hornjohnson}, we know that if there exists a positive vector~$\bds{\delta}$ and a positive constant~$\eta$ such that~${M_{\alpha, \beta} \bds{\delta} \leq \eta \bds{\delta}}$, then~${\rho(M_{\alpha, \beta}) \leq  \mn{M_{\alpha, \beta}}^{\bds{\delta}}_\infty \leq \eta}$, where~$\mn{\cdot}^{\bds{\delta}}_\infty$ is the matrix norm induced by the weighted max-norm~$\|\cdot\|^{\bds{\delta}}_\infty$, with respect to some positive vector~${\bds{\delta}}$. To this end, we first choose~${\eta := \left(1 - \alpha \beta \frac{\sigma_m^2}{{\kappa}} \right)}$, which is clearly less than~$1$. We next solve for a range of~${\alpha, \beta > 0}$ and for a positive vector~${\bds{\delta} = \left[\delta_1, \delta_2, \delta_3, \delta_4, \delta_5, \delta_6 \right]^\top}$ such that the inequalities in~${M_{\alpha, \beta} \bds{\delta} \leq \left(1 - \alpha \beta \frac{\sigma_m^2}{{\kappa}} \right) \bds{\delta}}$ hold element-wise. With~$M_{\alpha, \beta}$ in~Appendix~\ref{appA}, from the first and the fourth rows, we obtain
\begin{align}
\alpha \beta \frac{\sigma_m^2}{{\kappa}} + \alpha L \frac{\delta_3}{\delta_1} &\leq 1 - \lambda, \label{e1}\\
\alpha \beta \frac{\sigma_m^2}{{\kappa}} + \beta {L} \frac{\delta_6}{\delta_4} &\leq 1 - \lambda. \label{e4}
\end{align}
Similarly, from the second and the fifth rows, we obtain
\begin{align}
\beta {\mu} \delta_2 &\leq \delta_2 - \frac{{L}}{\sigma_m^2}  \left(L \delta_1 + \sigma_M \delta_5 \right), \label{e2}\\
\alpha \frac{\sigma_m^2}{{\kappa}} \delta_5 &\leq {\mu}\left( 1 - \frac{1}{c}\right)\delta_5 - \frac{{\mu} L}{c \sigma_M} \delta_1 \nonumber\\
&~~~ - \frac{{\mu}^2}{c \sigma_M^2} m_3 \delta_2 - \left(\frac{{\mu}}{c} + {L} \right) \delta_4. \label{e5}
\end{align}
Finally, form the third and the sixth rows, we obtain
\begin{align} \label{e3}
    &\left(\alpha \lambda L + \beta \frac{\lambda \sigma_M^2}{L}\right) \delta_1 + \left(\alpha m_1 + \beta m_2 \right) \delta_2 + \beta \lambda \sigma_M \delta_6 \nonumber\\
    &+ \left(\alpha \lambda L + \alpha \beta \frac{\sigma_m^2}{{\kappa}}\right) \delta_3 + \left(\alpha \lambda \sigma_M + \beta \lambda \sigma_M \right) \delta_4 \\
    &+ \left(\alpha \lambda \sigma_M + \beta \lambda \sigma_M \right) \delta_5 \leq \left(1 - \lambda \right) \delta_3 - \lambda \left(\delta_1 + \frac{\sigma_M}{L} \delta_4 \right), \nonumber
\end{align}
\begin{align} \label{e6}
    &\left(\alpha \lambda \sigma_M + \beta \lambda \sigma_M \right) \delta_1 + \left(\alpha m_4 + \beta m_5 \right) \delta_2 + \alpha \lambda \sigma_M \delta_3 \nonumber\\
    &+ \left(\alpha \frac{\lambda \sigma_M^2}{{L}} + \beta \lambda {L} \right) \delta_4 + \left(\alpha \frac{\lambda \sigma_M^2}{{L}} + \beta \lambda {L} \right) \delta_5 \\
    &+ \left(\beta \lambda L + \alpha \beta \frac{\sigma_m^2}{{\kappa}}\right) \delta_6 \leq \left( 1 - \lambda \right) \delta_6 - \lambda \left( \frac{\sigma_M}{{L}} \delta_1 + \delta_4 \right). \nonumber
\end{align}
We note that~\eqref{e2}--\eqref{e6} hold true for some feasible range of~$\alpha$ and~$\beta$ when their right hand sides are positive. Thus, we fix the elements of~$\bds{\delta}$ (independent of stepsizes) as
\begin{align*}
\delta_1 &= \tfrac{\sigma_M}{L}, \qquad \quad ~~\delta_2 = 4 l_2 \left[1
+ l_3 \right], \qquad \delta_3 = \tfrac{\lambda}{1 - \lambda} \tfrac{2 \sigma_M}{L},\\
\delta_4 &= \tfrac{c-1}{2(1+c {\kappa})}, \qquad \delta_5 = 2 l_3 \left[1 + 4 l_1 l_2 \left(1
+ l_3 \right) \right] + 1,\\
\delta_6 &= \tfrac{\lambda}{1 - \lambda} \sigma_M^2 \left(\tfrac{1}{L^2} + \tfrac{1}{\sigma_m^2} \right),
\end{align*}
where~${c > \frac{2 L^2}{\sigma_m^2} + \frac{2 \sigma_M^2 {\kappa}}{\sigma_m^2} + 1 }$ and 
\begin{align*}
l_1 := \tfrac{L}{\sigma_M} + \tfrac{\sigma_M}{{\mu}}, \quad l_2 := \tfrac{\sigma_M {L}}{d \sigma_m^2}, \quad l_3 := \tfrac{1}{c-1}, \quad  d = \tfrac{1}{1+2 l_1 l_2 l_3}.
\end{align*}
It can be verified that for the above choice of~$\bds{\delta}$, the right hand sides of~\eqref{e2}--\eqref{e6} are positive. Next we solve for the range of~$\alpha$ and~$\beta$. It can be verified that~\eqref{e1} and~\eqref{e4} are satisfied when
\begin{align*}
\alpha \leq \frac{ (1 - \lambda)^2}{4 \lambda \sigma_M}, \qquad \beta &\leq \frac{(1 - \lambda)^2}{4 \lambda \sigma_M^2} \left(\frac{c-1}{1 + c \kappa} \right)\left( \frac{ \sigma_m^2 L}{L^2 + \sigma_m^2} \right) \\
\text{and} \qquad \alpha \beta &\leq \frac{ {\kappa}( 1 - \lambda)}{2 \sigma_m^2}.
\end{align*}
Similarly, the relations~\eqref{e2} and~\eqref{e5} hold for
\begin{align*}
\alpha \leq \frac{{\kappa} {\mu}(c-1)}{2 c \sigma_m^2} \qquad \text{and} \qquad \beta \leq \frac{1}{4{\mu}}.
\end{align*}
Finally, it can be verified that~\eqref{e3} and~\eqref{e6} hold when
\begin{align*}
    \alpha &\leq \min \bigg \{ \frac{d \sigma_m^2}{384 L^3}, \frac{d \sigma_m^2}{384 \sigma_M^2 \kappa L}, \frac{d \mu}{384 \sigma_M^3},\\
    &~~~~\frac{1 - \lambda}{48 L} \left( \frac{1}{2 l_3 \left[1 + 4 l_1 l_2 \left(1
    + l_3 \right) \right] + 1} \right), \\ &~~~~\frac{{L}(1-\lambda)}{24 \sigma_M^2} \left( \frac{1}{2 l_3 \left[1 + 4 l_1 l_2 \left(1
    + l_3 \right) \right] + 1} \right) \bigg \}, \\
    \beta &\leq \min \bigg \{ \frac{\sigma_m^2 (1 - \lambda)^2}{192 \sigma_M^2 L}, \frac{L (1 - \lambda)^2}{48 \sigma_M^2},
    \frac{d}{382 \kappa L},\\
    &~~~~\frac{1}{24 L} \left( \frac{1}{2 l_3 \left[1 + 4 l_1 l_2 \left(1
    + l_3 \right) \right] + 1 } \right) \bigg \},
\end{align*}
and
\begin{align*}
\alpha \beta \leq  \min \bigg\{ \frac{(1 - \lambda) {\kappa}}{48 \sigma_m^2}, \frac{(1 - \lambda) {\kappa} L^2 }{48 \sigma_m^4} \bigg \}. 
\end{align*}
The lemma follows by simplifying all the~$\alpha$ and~$\beta$ bounds for some large enough~${c}$ with~${\eta = \left(1 - \alpha \beta \frac{\sigma_m^2}{{\kappa}} \right)}$. 
% \qed
\end{proof}
% \subsubsection{Proof of Lemma~\ref{lem2}}
% We use the fact that if~${G_{\alpha, \beta} \bds{\delta} < \eta \bds{\delta}}$ for some positive vector~${\bds{\delta} > \mb{0}_6}$ and~${\eta = 1 - \alpha \beta \frac{\sigma_m^2}{{\kappa}} }$
% The complete proof can be found in the technical report~\cite{gda_arxiv}. 
% \vspace{-0.8cm}
The above lemma shows that the spectral radius of the system matrix~${M_{\alpha, \beta}}$ is less than or equal to a positive constant~${\eta<1}$ for appropriate stepsizes~$\alpha$ and~$\beta$. We emphasize that the proof of Lemma~\ref{lem2} does not follow the conventional strategies used in the literature on distributed optimization~\cite{pushsaga, ABSAGA,hornjohnson}. We need to carefully select an expression of~${\eta}$ and then find appropriate bounds on both the stepsizes ($\alpha$ and~$\beta$) to ensure convergence. Using the two lemmas above, we are now in a position to prove Theorem~\ref{th1}. It is noteworthy that~${N_{\alpha, \beta, k}}$ decays faster than~${M_{\alpha, \beta}}$ as it can be verified that~${\lambda \leq \rho (M_{\alpha, \beta})}$. We now show that~${\| \mb{u}^{k}\| \ra 0}$ and prove Theorem~\ref{th1}.

% Next, we provide a useful lemma to show that~${\| \mb{u}^{k}\| \ra 0}$.

% The proof of above Lemma~\ref{lem2} follows the arguments provided in~\cite{pushsaga,hornjohnson} and can be found in the technical report~\cite{gda_arxiv}. Using the two lemmas above, we are now in a position to prove Theorem~\ref{th1}.

% \begin{lem}

% \end{lem}

\subsubsection{Proof of Theorem~\ref{th1}} We first rewrite the LTI system dynamics described in Lemma~\ref{lem1} recursively as
\begin{align} \label{recur_eq}
    \mb{u}^{k} &\leq M_{\alpha,\beta}^k \mb{u}^{0} +  \sum_{r=0}^{k-1} M_{\alpha,\beta}^{k-r-1} N_{\alpha,\beta,k} \mb{s}^{r}.
\end{align}
We now take the norm on both sides such that for some positive constants~${\omega_1, \omega_2, \omega_3}$ and~${\omega_4}$,~\eqref{recur_eq} can be written as
% The LTI system dynamics described in Lemma~\ref{lem1} can be recursively re-evaluated. After taking norm, for some positive constants~${\omega_1, \omega_2, \omega_3}$ and~${\omega_4}$,~\eqref{sys_eq} can be written as
\begin{align*}
    \|\mb{u}^{k}\| &\leq \|M_{\alpha,\beta}^k \mb{u}^{0}\| +  \sum_{r=0}^{k-1} \|M_{\alpha,\beta}^{k-r-1} N_{\alpha,\beta,k} \mb{s}^{r}\| \\
    &\leq \omega_1 \eta^k + \omega_2 \eta^k \sum_{r=0}^{k-1} \|\mb{s}^{r}\|,
\end{align*}
where~${\|\mb{s}^{r}\| \leq \omega_3 \| \mb{u}^k \| + \omega_4 \| \mb{x}^* \| + \omega_5 \| \mb{y}^* \|}$, with the help of some arbitrary norm equivalence constants. It can be verified that for~${a:=\omega_4 \| \mb{x}^* \| + \omega_5 \| \mb{y}^*\|}$,
\begin{align*}
    \|\mb{u}^{k}\| &\leq \left( \omega_1 + k a + \omega_2 \omega_3 \sum_{r=0}^{k-1} \|\mb{u}^{r}\| \right) \eta^k.
\end{align*}
Let~${b_k:= \sum_{r=0}^{k-1} \|\mb{u}^{r}\|}$,~${c_k:= (\omega_1 + k a) \eta^k}$ and~${d_k:= \omega_2 \omega_3 \eta^k}$. Then the above can be re-written as
\begin{align*}
    \|\mb{u}^{k}\| = b_{k+1} - b_{k}
    &\leq \left( \omega_1 + k a + \omega_2 \omega_3 b_k \right) \eta^k, \\
    \iff b_{k+1} &\leq (1 + d_k) b_k + c_k.
\end{align*}
For non-negative sequences~${\{b_k\}, \{c_k\}}$ and~${\{ d_k \}}$ related as~${b_{k+1} \leq (1 + d_k) b_k + c_k}, {\forall k}$, such that~${\sum_{k=0}^\infty c_k < \infty}$ and~${\sum_{k=0}^\infty d_k < \infty}$, we have that the sequence~${\{b_k\}}$ converges and is this bounded~\cite{polyak}. Therefore,~${\forall \nu \in (\eta, 1)}$, we have
\begin{align*}
    \lim_{k \ra \infty} \frac{\| \mb{u}^k \|}{\nu^k} \leq \lim_{k \ra \infty} \frac{\left(\omega_1 + k a + \omega_2 \omega_3 b_k \right) \eta^k }{\nu^k} = 0,
\end{align*}
and there exists a~${\psi > 0}$, such that~${\| \mb{u}^k \| \leq \psi(\eta + \xi)^k}$ for all~${k}$, where~${\xi>0}$ is an arbitrarily small constant. To achieve an~${\epsilon}$-accurate solution, we need
\begin{align*}
    \|\mb{x}^k - \mb{1}_n \otimes \mb{x}^* \| + \|\mb{y}^k - \mb{1}_n \otimes \mb{y}^* \| &\leq \epsilon, \\
    \impliedby \|\mb{u}^k\| \leq e^{-(1 - (\eta + \xi))k} \theta &\leq \epsilon,
\end{align*}
and the theorem follows.
\qed
% It is of importance to note that~${H_{\alpha, \beta}}$ decays faster than~${G_{\alpha, \beta}}$ as it can be verified that~${\lambda \leq \rho (G_{\alpha, \beta})}$, and Theorem~\ref{th1} follows. 
% \qed
% \newpage
\subsection{Convergence of~$\GDAl$}
We now establish the convergence of~$\GDAl$ under the corresponding set of assumptions. It can be verified that the LTI system for~$\GDAl$ is similar to the one described in Lemma~\ref{lem1} except that the non-zero elements in~${N_{\alpha, \beta, k}}$ are replaced by~$\alpha \tau$ at the $(2,1)$ location and~$\beta \tau$ at the~$(1,5)$ location (named as~${\wt{N}_{\alpha, \beta}}$).
% \begin{align*}
%   \wt{H}_{\alpha,\beta} \!&:=\! \left[ {\begin{array}{c c c c c c}
%     0 &0 &0 &0 &0 &0 \\
%     0 &\alpha \nu &0 &0 &0 &0 \\
%     0 &0 &0 &0 &0 &0 \\
%     0 &0 &0 &0 &0 &0 \\
%     \beta \nu &0 &0 &0 &0 &0 \\
%     0 &0 &0 &0 &0 &0
%   \end{array} } \right],
% \end{align*}
In the following lemma, we consider the convergence of~${\GDAl}$ under least assumptions, i.e., when~${Pi}$'s are not necessarily identical.

\subsubsection{Proof of Theorem~\ref{th2}}
Consider~$\GDAl$ under Assumptions~\ref{sm_scc},~\ref{cp_rank}, and~\ref{doub_stoc}. Given the stepsizes~${\alpha \in (0, \ol{\alpha}]}$ and~${\beta \in (0, \ol{\beta}]}$, with~$\ol{\alpha}$ and~$\ol{\beta}$ defined in Theorem~\ref{th1}, we have
\begin{align} \label{sys_eq_2}
    \mb{u}^{k} &\leq M_{\alpha,\beta}^k \mb{u}^{0} +  \sum_{r=0}^{k-1} M_{\alpha,\beta}^{k-r-1} \wt{N}_{\alpha,\beta} \mb{s}^{k}.
\end{align}
We have already established the fact that~${\rho(M_{\alpha, \beta}) \leq \eta < 1}$ (see Lemma~\ref{lem2}). Therefore, the first term disappears exponentially, and the asymptotic response is
\begin{align} \label{sys_eqq}
    \limsup_{k \ra \infty} \mb{u}^{k} &\leq (I - M_{\alpha,\beta})^{-1} \wt{N}_{\alpha,\beta} \mb{s},
\end{align}
where~${\mb{s} := \left[ \sup_k \| \mb{x}^k \|, \sup_k \| \mb{y}^k \|, 0, 0, 0, 0 \right]^\top }$. It is noteworthy that the only two nonzero terms in the vector~${\wt{N}_{\alpha,\beta,k} \mb{s}}$ are controllable by the stepsizes~$\alpha$ or~$\beta$ and Theorem~\ref{th2} follows.
% Therefore, the steady-state error can be controlled by the stepsizes and Theorem~\ref{th2} follows.
\qed

Next we provide the proof of Theorem~\ref{th3}, which assumes that each node has the same coupling matrix, and establishes convergence of~$\GDAl$.
% \begin{theorem} \label{gdal_2}
\subsubsection{Proof of Theorem~\ref{th3}}
Consider~$\GDAl$ under Assumptions~\ref{sm_scc},~\ref{cp_rank}, and~\ref{doub_stoc}, with identical~$P_i$'s. It can be verified that for stepsizes~${\alpha \in (0, \ol{\alpha}]}$, and~${\beta \in (0, \ol{\beta}]}$, the LTI system described in~\eqref{sys_eq} reduces to
\begin{align} \label{sys_eq_3}
    \mb{u}^{k+1} &\leq M_{\alpha,\beta} \mb{u}^{k},
\end{align}
because~${\tau=0}$. The theorem thus since~${\rho(M_{\alpha, \beta})<1}$ from Lemma~\ref{lem2}.
\qed

\subsection{$\GDAl$ for quadratic problems}
We now consider~$\GDAl$ for quadratic problems where the coupling matrices are not necessarily identical at each node. We show that~$\GDAl$ converges to the unique saddle point without needing consensus. We now define the corresponding LTI system in the following lemma.

% Even with unidentical coupling matrices, we show that for the quadratic case, a consensus on Pi's is not required to converge to the unique saddle point and thus the applicable algorithm is~$\GDAl$ that does not implement a consensus. We first establish the corresponding LTI system in the following lemma.

% Now we present a lemma for~$\GDAl$ assuming the non-coupled functions to be quadratic and the coupling matrices are not necessarily identical at each node, and then use Lemma~\ref{semi_simple_eig} to establish linear convergence of~$\GDAl$.

\begin{lem}[$\GDAl$ for quadratic problems] \label{gdal_quad}
Consider Problem~$\P$ under Assumptions~\ref{quad},~\ref{cp_rank}, and~\ref{doub_stoc} (with different $P_i$'s at the nodes). 
% Let~${g_i(\mb x) = \mb x^\top Q_i\mb x + \mb{q}_i^\top \mb x + {q}_i}$, and~${h_i(\mb y) = \mb y^\top R_i\mb y + \mb{r}_i^\top \mb y + {r}_i }$ such that ${\mb{q}_i \in \R^{p_x}}$, ${\mb{r}_i \in \R^{p_y}}$, ${q_i, r_i \in \R}$,~${Q_i \in \R^{p_x \mt p_x}}$,~${R_i \in \R^{p_y \mt p_y}}$,~$\forall i$, and~${\ol{Q} := \frac{1}{n} \sum_{i=1}^n Q_i}$, and~${\ol{R} := \frac{1}{n} \sum_{i=1}^n R_i}$ are positive definite. 
Then the LTI dynamics governing~$\GDAl$ are defined by the following
\begin{align} \label{sys_eq_quad}
    \wt{\mb{u}}^{k+1} &= \wt{M}_{\alpha,\beta} \wt{\mb{u}}^{k},
\end{align}
where
\begin{align*}
\wt{\mb{u}}^{k} \!&:=\!  \left[ \begin{array}{c}
    \mb{x}^{k} - W_1^\infty \mb{x}^{k} \\
    \ol{\mb{x}}^{k} - \mb{x}^* \\
    \mb{q}^{k} - W_1^\infty \mb{q}^{k} \\
    \mb{y}^{k} - W_2^\infty \mb{y}^{k} \\
    \ol{\mb{y}}^{k} - \mb{y}^* \\
    \mb{w}^{k} - W_2^\infty \mb{w}^{k}
  \end{array}\right],
\end{align*} 
and the system matrix~$\wt{M}_{\alpha,\beta}$ is defined in Appendix~\ref{appB}. 
\end{lem}
Lemma~\ref{gdal_quad} provides an exact analysis of $\GDAl$ for quadratic problems. The derivation of above system follows similar arguments as provided in Lemma~\ref{lem1} without the use of norm inequalities.

Our aim now is to establish linear convergence of $\GDAl$ to the exact saddle point. To proceed, we set~${\beta = \alpha}$, which makes~${\wt{M}_{\alpha, \beta}}$ as a function of~${\alpha}$ only and define~${\wt{M}_{\alpha} := \wt{M}_{\alpha, \beta=\alpha}}$ and~${\wt{M}_{0} := \wt{M}_{\alpha=0}}$ leading to
\begin{align*}
\wt{M}_{0} \!&:=\! \left[ {\begin{array}{c c c c c c}
    \ol{W}_1 &O &O &O &O &O \\
    O &I_{p_x} &O &O &O &O \\
    \mt &O &\ol{W}_1 &\mt &O &O \\
    O &O &O &\ol{W}_2 &O &O \\
    O &O &O &O &I_{p_y} &O \\
    \mt &O &O &\mt &O & \ol{W}_2
  \end{array} } \right],
\end{align*}
where~$O$ are zero matrices of appropriate dimensions and~`${\mt}$' are ``don't care terms" that do not affect further analysis. Using Schur's Lemma for determinant of block matrices, we note that, for the given structure of~${\wt{M}_{0}}$, the eigenvalues of~${\wt{M}_{0}}$ are the eigenvalues of the diagonal block matrices. Furthermore, we know that~${\rho(W-W^\infty)<1}$, which implies that~${\rho(\ol{W}_1)<1}$ and~${\rho(\ol{W}_2)<1}$. Therefore,~${\rho(\wt{M}_{0})=1}$ and~${\wt{M}_{0}}$ has~${p:=p_x+p_y}$ semi-simple eigenvalues. We would like to show that for sufficiently small positive stepsizes and~${\beta = \alpha}$, all eigenvalues of~$1$ decrease and thus the spectral radius of~${\wt{M}_{\alpha,\beta}}$ becomes less than~$1$. It is noteworthy that the analysis in the existing literature on distributed optimization~\cite{AB, pushsaga_icassp} is limited to a simple eigenvalue and thus cannot be directly extended to our case. We provide the following lemma to establish the change in the semi-simple eigenvalues with respect to~${\alpha}$.

\begin{lem} \label{semi_simple_eig}
\cite{semi_simp_eig, DSVM_Reza, kai} Consider an~${n \mt n}$ matrix~$M_\alpha$ which depends smoothly on a real parameter~${\alpha \geq 0}$. Fix
${l \in [1, n]}$ and let ${\lambda_1 = \cdots = \lambda_l}$ be a semi-simple eigenvalue of~$M_0$, with (linearly independent) right eigenvectors ${\mb{y}_1, \cdots, \mb{y}_l}$ and (linearly independent) left eigenvectors ${\mb{z}_1, \cdots, \mb{z}_l}$ such that
\begin{align*}
    \left[ {\begin{array}{c}
    \mb{z}_1\\
    \vdots\\
    \mb{z}_l
  \end{array} } \right] 
  \left[ {\begin{array}{c c c}
    \mb{y}_1 \quad \cdots \quad \mb{y}_l
  \end{array} } \right] = I_{l}.
\end{align*}

\noindent Denote by~${\lambda_i(\alpha)}$ the eigenvalues of~$M_\alpha$ corresponding to~${\lambda_i}$,~${i \in [1, l]}$, as a function of~$\alpha$. Then the derivatives~${d \lambda_i(\alpha)/d\alpha}$ exist, and ~${d \lambda_i(\alpha)/d\alpha|_{\alpha=0}}$ is given by the eigenvalues~$[\lambda_S]_i$ of the following~${l \mt l}$ matrix
\begin{align*}
    S := \left[ {\begin{array}{c c c}
    \mb{z}_1^\top M' \mb{y}_1 &\cdots & \mb{z}_1^\top M' \mb{y}_l \\
    \vdots &\ddots & \vdots \\
    \mb{z}_l^\top M' \mb{y}_1 &\cdots & \mb{z}_l^\top M' \mb{y}_l 
  \end{array} } \right],
\end{align*}
where~${M':= d M_\alpha/d\alpha|_{\alpha=0}}$. Furthermore, the eigenvalue ${\lambda_i(\alpha)}$ under perturbation of the parameter~${\alpha}$ is given by
% the eigenvalues are
\begin{align*} 
    \lambda_i(\alpha) &= \lambda_i(0) + \alpha [\lambda_{S}]_i + o (\alpha), \qquad \forall i \in [1, l].
\end{align*}
\end{lem}

% \begin{lem} \label{circle_line}
% A straight line intersecting a circle is either tangent or a secant to the circle. 
% \end{lem}

With the help of Lemma~\ref{semi_simple_eig}, we now prove Theorem~\ref{th4}.
\subsubsection{Proof of Theorem 4}
% We first need to show that~$\wt{M}_0$ has~$p$ semi-simple eigenvalues.

% It can be verified that~$\wt{M}_0$ has~$p$ semi-simple eigenvalues, see~\cite{gda_arxiv} for details. 
We note that~${\wt{M}_\alpha:= \wt{M}_0 + \alpha M'}$, where~${M':= d \wt{M}_{\alpha}/d \alpha |_{\alpha=0}}$. Furthermore, the left eigenvectors corresponding to~$p$ semi-simple eigenvalues are the rows of the matrix~$L$ and the right eigenvectors are the columns of the matrix~$R$, as defined below.
\begin{align*}
    L := \left[ {\begin{array}{c c c c c c}
    O &I_{p_x} &O &O &O &O\\
    O &O &O &O &I_{p_y} &O\\
  \end{array} } \right] \quad \mbox{and} \quad
  R := L^\top.
\end{align*}
We would like to ensure that the semi-simple eigenvalues of~${\wt{M}_0}$ are forced to move inside the unit circle as~${\alpha}$ increases. Using Lemma~\ref{semi_simple_eig}, it can be established that the derivatives~${d \lambda_i(\alpha)/d\alpha|_{\alpha=0}}$ exist for all~${i \in [1, p]}$ and are the eigenvalues~$[\lambda_S]_i$ of
\begin{align*}
    S := \left[ {\begin{array}{c c}
    -\ol{Q} & -\ol{P}^\top \\
    \ol{P} & -\ol{R}
  \end{array} } \right].
\end{align*}
% We note that~$S$ 
Next we define~${a_i := \mbox{Re} \left( [\lambda_S]_i \right)}$ and~${b_i := \mbox{Im} \left( [\lambda_S]_i \right)}$, ${\forall i \in [1, p]}$. 
% We note that the symmetric part~$\mbox{Sym}(\cdot)$ of~$S$ is
% \begin{align*}
%     \mbox{Sym}(S) = \frac{1}{2} \left(S + S^\top \right) = - \frac{1}{2} \left[ {\begin{array}{c c}
%     \ol{Q} + \ol{Q}^\top & O \\
%     O & \ol{R} + \ol{R}^\top
%   \end{array} } \right].
% \end{align*}
% We know from that a matrix is negative definite if and only if its symmetric part is negative definite~\cite{symm_mat}, i.e., ${\forall i \in [1, p]}$,~${a_i < 0}$. 
From Theorem 3.6 in~\cite{golub}, we know that~$-S$ is positive stable if~${(\ol{Q} + \ol{Q}^\top)}$ is positive definite and~${(\ol{R} + \ol{R}^\top)}$ is positive semi-definite, i.e.,~$\forall i \in [1, p], a_i < 0$.
Furthermore, we know from Lemma~\ref{semi_simple_eig} that
\begin{align} \label{eig}
    \lambda_i(\alpha) &= \lambda_i(0) + \alpha [\lambda_{S}]_i + o (\alpha),
\end{align}
see Theorem 2.7 in~\cite{semi_simp_eig} for details. For sufficiently small stepsize~${\alpha > 0}$, the term~${o (\alpha)}$ can be made arbitrarily small as it contains higher order~${\alpha}$ terms. Therefore, we can re-write~\eqref{eig} as
\begin{align*}
    \lambda_i(\alpha) &= 1 + \alpha [\lambda_{S}]_i = 1 + \alpha (a_i + j b_i), \qquad \forall i \in [1, p].
\end{align*}
Since the real parts of~${[\lambda_S]_i}$ are~${a_i < 0}$, the semi-simple eigenvalues would move towards the direction of the secants of the unit circle for any~${b_i}$, see Fig.~\ref{circ}. We thus obtain that~${\rho (\wt{M}_{\alpha}) < 1 }$ for sufficiently small stepsize~${\alpha > 0}$ and Theorem~\ref{th4} follows. $\hfill \qed$
\begin{figure}[!ht]
\centering
\includegraphics[width=2.5in]{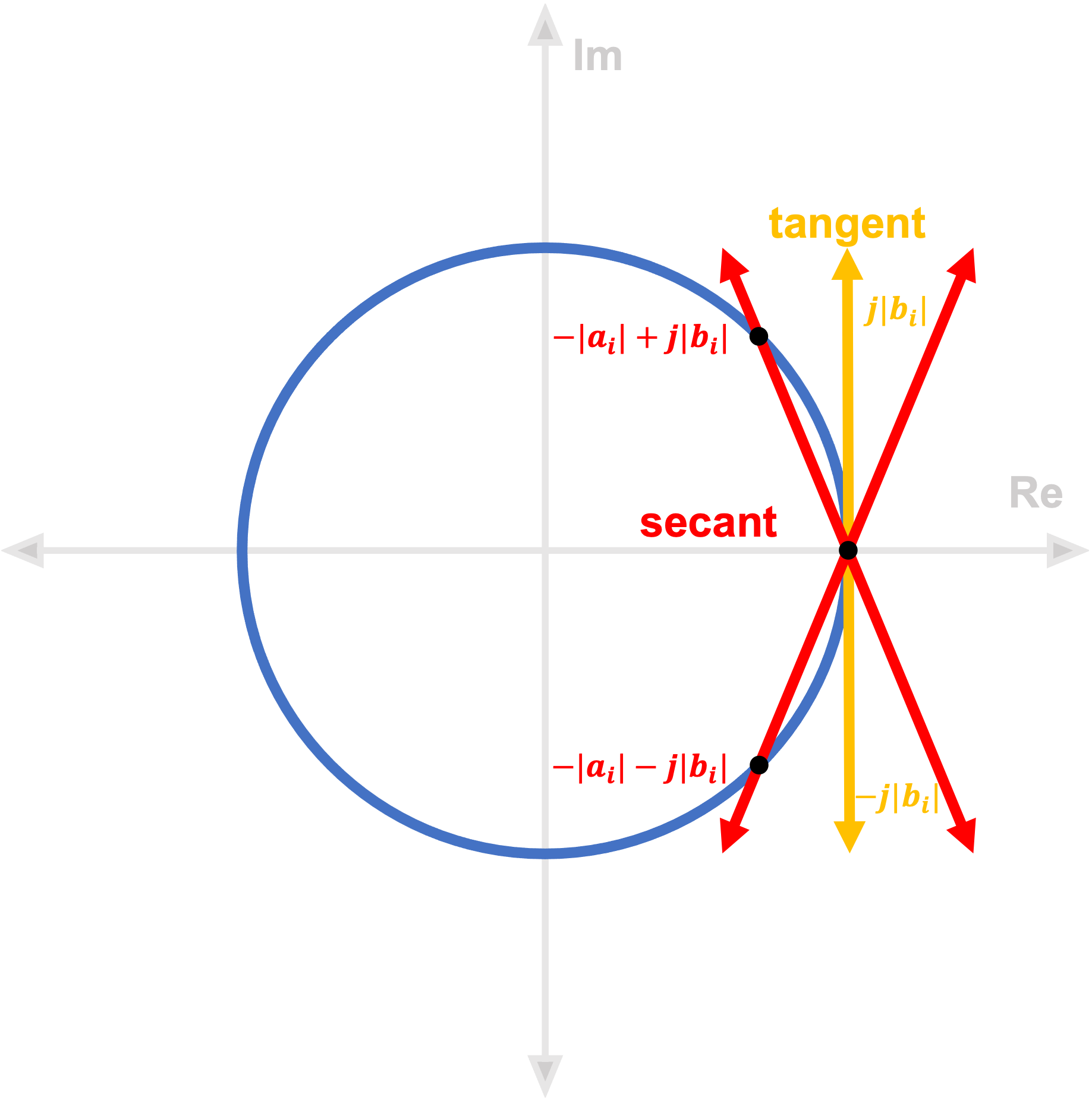}
\caption{Tangent lines intersect a circle at a single point. Secant lines intersect a circle at two points.}
\label{circ}
\end{figure}
% Therefore,~$\mc{S}(S)$ 
% is negative definite because~$\ol{Q}$ and~$\ol{R}$ are positive definite. 
% This makes the derivatives~${d \lambda_i(\alpha)/d\alpha|_{\alpha=0} < 0}$. Therefore, the spectral radius~${\rho (\wt{M}_{\alpha}) < 1 }$ as the stepsize~$\alpha$ slightly increases from zero and Theorem~\ref{th4} follows.

\section{Conclusion} \label{conc}
In this paper, we describe first-order methods to solve distributed saddle point problems of the form: ${\min_{\mb{x}} \max_{\mb{y}} \left\{ G(\mb x) + \mb \langle \mb y, \ol{P} \mb x \rangle - H(\mb y) \right\}}$, which has many practical applications as described earlier in the paper. In particular, we assume that the underlying data is distributed over a strongly connected network of nodes such that~${G(\mb{x}):= \tfrac{1}{n} \sum_{i=1}^n g_i(\mb{x})}$, ${H(\mb{y}):=\tfrac{1}{n} \sum_{i=1}^n h_i(\mb{y})}$, and~${\ol{P}:= \tfrac{1}{n} \sum_{i=1}^n P_i}$, where the constituent functions~$g_i$,~$h_i$, and local coupling matrices~$P_i$ are private to each node~$i$. 
% Assuming each~$g_i(\cdot)$ and~$h_i(\cdot)$ to be smooth, the global~$G(\cdot)$ to be convex, while the global~$H(\cdot)$ is strongly convex. 
Under appropriate assumptions, we show that~$\GDA$ converges linearly to the unique saddle point of strongly concave-convex problems. We further provide explicit $\epsilon$-complexities of the underlying algorithms and characterize a regime in which the convergence is network-independent. To reduce the communication complexity of~$\GDA$, we propose a lighter (communication-efficient) version~$\GDAl$ that does not require consensus on local~$P_i$'s and analyze~$\GDAl$ in various relevant scenarios. Finally, we illustrate the convergence properties of~$\GDA$ through numerical experiments. 

\bibliography{sample,proc}

\begin{thebibliography}{10}

\bibitem{hall}
E.~L. Hall, J.~J. Hwang, and F.~A. Sadjadi,
\newblock ``{Computer Image Processing And Recognition},''
\newblock in {\em Optics in Metrology and Quality Assurance}, Harvey~L. Kasdan,
  Ed. International Society for Optics and Photonics, 1980, vol. 0220, pp. 2 --
  10, SPIE.

\bibitem{golub}
M.~Benzi, G.~H. Golub, and J.~Liesen,
\newblock ``Numerical solution of saddle point problems,''
\newblock {\em Acta Numerica}, vol. 14, pp. 1–137, 2005.

\bibitem{Goodfellow}
I.~Goodfellow, J.~Pouget-Abadie, M.~Mirza, B.~Xu, D.~Warde-Farley, S.~Ozair,
  A.~Courville, and Y.~Bengio,
\newblock ``Generative adversarial nets,''
\newblock in {\em Advances in Neural Information Processing Systems},
  Z.~Ghahramani, M.~Welling, C.~Cortes, N.~Lawrence, and K.~Q. Weinberger, Eds.
  2014, vol.~27, Curran Associates, Inc.

\bibitem{sinha}
A.~Sinha, H.~Namkoong, and J.~Duchi,
\newblock ``Certifiable distributional robustness with principled adversarial
  training,''
\newblock in {\em International Conference on Learning Representations}, 2018.

\bibitem{jordan}
T.~Lin, C.~Jin, and M.~I. Jordan,
\newblock ``Near-optimal algorithms for minimax optimization,'' 2021.

\bibitem{jordan_GDA}
T.~Lin, C.~Jin, and M.~I. Jordan,
\newblock ``On gradient descent ascent for nonconvex-concave minimax
  problems,'' 2021.

\bibitem{stokes}
T.~Liang and J.~Stokes,
\newblock ``Interaction matters: A note on non-asymptotic local convergence of
  generative adversarial networks.,''
\newblock {\em CoRR}, vol. abs/1802.06132, 2018.

\bibitem{forero2010consensus}
P.~A. Forero, A.~Cano, and G.~B. Giannakis,
\newblock ``Consensus-based distributed support vector machines,''
\newblock {\em Journal of Machine Learning Research}, vol. 11, no. May, pp.
  1663--1707, 2010.

\bibitem{6119236}
J.~F.~C. Mota, J.~M.~F. Xavier, P.~M.~Q. Aguiar, and M.~P\"{u}schel,
\newblock ``Distributed basis pursuit,''
\newblock {\em IEEE Trans. on Signal Process.}, vol. 60, no. 4, pp. 1942--1956,
  Apr. 2012.

\bibitem{tutorial_nedich}
A.~Nedi{\'c}, A.~Olshevsky, and M.~G. Rabbat,
\newblock ``Network topology and communication-computation tradeoffs in
  decentralized optimization,''
\newblock {\em Proceedings of the IEEE}, vol. 106, no. 5, pp. 953--976, 2018.

\bibitem{DOPT_survey_yang}
T.~Yang, X.~Yi, J.~Wu, Y.~Yuan, D.~Wu, Z.~Meng, Y.~Hong, H.~Wang, Z.~Lin, and
  K.~H. Johansson,
\newblock ``A survey of distributed optimization,''
\newblock {\em Annual Reviews in Control}, 2019.

\bibitem{du_quad}
S.~S. Du, J.~Chen, L.~Li, L.~Xiao, and D.~Zhou,
\newblock ``Stochastic variance reduction methods for policy evaluation,''
\newblock in {\em Proceedings of the 34th International Conference on Machine
  Learning}, Doina Precup and Yee~Whye Teh, Eds. 06--11 Aug 2017, vol.~70 of
  {\em Proceedings of Machine Learning Research}, pp. 1049--1058, PMLR.

\bibitem{du}
S.~S. Du and W.~Hu,
\newblock ``Linear convergence of the primal-dual gradient method for
  convex-concave saddle point problems without strong convexity,'' 2019.

\bibitem{WW_scutari}
A.~Beznosikov, G.~Scutari, A.~Rogozin, and A.~Gasnikov,
\newblock ``Distributed saddle-point problems under similarity,'' 2021.

\bibitem{xianWWSGDA}
W.~Xian, F.~Huang, Y.~Zhang, and H.~Huang,
\newblock ``A faster decentralized algorithm for nonconvex minimax problems,''
\newblock in {\em Advances in Neural Information Processing Systems},
  A.~Beygelzimer, Y.~Dauphin, P.~Liang, and J.~Wortman Vaughan, Eds., 2021.

\bibitem{Von_Neumann}
J.~V. Neumann and O.~Morgenstern,
\newblock {\em Theory of Games and Economic Behavior},
\newblock Princeton University Press, 1944.

\bibitem{Basar_Olsder}
T.~Başar and G.J. Olsder,
\newblock ``Dynamic non-cooperative game theory,''
\newblock vol. 160, 01 1999.

\bibitem{mokhtari}
A.~Mokhtari, A.~Ozdaglar, and S.~Pattathil,
\newblock ``A unified analysis of extra-gradient and optimistic gradient
  methods for saddle point problems: Proximal point approach,'' 2019.

\bibitem{OGDA_contact}
Y.~Malitsky and M.~K. Tam,
\newblock ``A forward-backward splitting method for monotone inclusions without
  cocoercivity,''
\newblock {\em SIAM Journal on Optimization}, vol. 30, no. 2, pp. 1451--1472,
  2020.

\bibitem{gdfree_VPoor}
T.~Xu, Z.~Wang, Y.~Liang, and H.~V. Poor,
\newblock ``Gradient free minimax optimization: Variance reduction and faster
  convergence,'' 2021.

\bibitem{bo}
I.~Bogunovic, J.~Scarlett, S.~Jegelka, and V.~Cevher,
\newblock ``Adversarially robust optimization with gaussian processes,''
\newblock in {\em NeurIPS}, 2018.

\bibitem{ga}
J.W. Herrmann,
\newblock ``A genetic algorithm for minimax optimization problems,''
\newblock in {\em Proceedings of the 1999 Congress on Evolutionary
  Computation-CEC99 (Cat. No. 99TH8406)}, 1999, vol.~2, pp. 1099--1103 Vol. 2.

\bibitem{pso}
E.~Laskari, K.~Parsopoulos, and M.~Vrahatis,
\newblock ``Particle swarm optimization for minimax problems,''
\newblock 02 2002, vol.~2, pp. 1576--1581.

\bibitem{DSGD_nedich}
S.~S. Ram, A.~Nedi{\'c}, and V.~V. Veeravalli,
\newblock ``Distributed stochastic subgradient projection algorithms for convex
  optimization,''
\newblock {\em Journal of Optimization Theory and Applications}, vol. 147, no.
  3, pp. 516--545, 2010.

\bibitem{DGD_Kar}
S.~Kar, J.~M.~F. Moura, and K.~Ramanan,
\newblock ``Distributed parameter estimation in sensor networks: Nonlinear
  observation models and imperfect communication,''
\newblock {\em IEEE Transactions on Information Theory}, vol. 58, no. 6, pp.
  3575--3605, 2012.

\bibitem{Diffusion_Chen}
J.~Chen and A.~H. Sayed,
\newblock ``Diffusion adaptation strategies for distributed optimization and
  learning over networks,''
\newblock {\em IEEE Transactions on Signal Processing}, vol. 60, no. 8, pp.
  4289--4305, 2012.

\bibitem{DEXTRA}
C.~{Xi} and U.~A. {Khan},
\newblock ``{DEXTRA}: A fast algorithm for optimization over directed graphs,''
\newblock {\em IEEE Transactions on Automatic Control}, vol. 62, no. 10, pp.
  4980--4993, Oct. 2017.

\bibitem{add-opt}
C.~Xi, R.~Xin, and U.~A. Khan,
\newblock ``{ADD-OPT}: Accelerated distributed directed optimization,''
\newblock {\em IEEE Transactions on Automatic Control}, vol. 63, no. 5, pp.
  1329--1339, 2017.

\bibitem{AB}
R.~Xin and U.~A. Khan,
\newblock ``A linear algorithm for optimization over directed graphs with
  geometric convergence,''
\newblock {\em IEEE Control Systems Letters}, vol. 2, no. 3, pp. 315--320,
  2018.

\bibitem{ABN}
R.~Xin, D.~Jakoveti\'{c}, and U.~A. Khan,
\newblock ``{Distributed Nesterov gradient methods over arbitrary graphs},''
\newblock {\em IEEE Signal Processing Letters}, vol. 26, no. 18, pp.
  1247--1251, Jun. 2019.

\bibitem{proxPDA_Hong}
M.~Hong, D.~Hajinezhad, and M.~Zhao,
\newblock ``{Prox-PDA:} the proximal primal-dual algorithm for fast distributed
  nonconvex optimization and learning over networks,''
\newblock in {\em Proceedings of the 34th International Conference on Machine
  Learning}, 2017, pp. 1529--1538.

\bibitem{ADMM_hong}
M.~Hong, Z.~Luo, and M.~Razaviyayn,
\newblock ``Convergence analysis of alternating direction method of multipliers
  for a family of nonconvex problems,''
\newblock {\em SIAM J. on Optim.}, vol. 26, no. 1, pp. 337--364, 2016.

\bibitem{SUCAG}
H.~Wai, N.~M. Freris, A.~Nedi\'{c}, and A.~Scaglione,
\newblock ``{SUCAG:} stochastic unbiased curvature-aided gradient method for
  distributed optimization,''
\newblock in {\em Proc. IEEE Conf. Decis. Control}, 2018, pp. 1751--1756.

\bibitem{Async_NN_Wei}
F.~Mansoori and E.~Wei,
\newblock ``Superlinearly convergent asynchronous distributed network newton
  method,''
\newblock in {\em Proc. IEEE Conf. Decis. Control}, 2017, pp. 2874--2879.

\bibitem{dual_optimal_ICML}
K.~Scaman, F.~Bach, S.~Bubeck, Y.~T. Lee, and L.~Massouli{\'e},
\newblock ``Optimal algorithms for smooth and strongly convex distributed
  optimization in networks,''
\newblock in {\em 34th International Conference on Machine Learning}, 2017, pp.
  3027--3036.

\bibitem{GT_CDC}
J.~Xu, S.~Zhu, Y.~C. Soh, and L.~Xie,
\newblock ``Augmented distributed gradient methods for multi-agent optimization
  under uncoordinated constant stepsizes,''
\newblock in {\em 54th IEEE Conference on Decision and Control}, 2015, pp.
  2055--2060.

\bibitem{NEXT}
P.~D. Lorenzo and G.~Scutari,
\newblock ``{NEXT:} in-network nonconvex optimization,''
\newblock {\em IEEE Transactions on Signal and Information Processing over
  Networks}, vol. 2, no. 2, pp. 120--136, 2016.

\bibitem{DSGD_NIPS}
X.~Lian, C.~Zhang, H.~Zhang, C.~Hsieh, W.~Zhang, and J.~Liu,
\newblock ``Can decentralized algorithms outperform centralized algorithms? {A}
  case study for decentralized parallel stochastic gradient descent,''
\newblock in {\em 30th Advances in Neural Information Processing Systems},
  2017, pp. 5330--5340.

\bibitem{D2}
H.~Tang, X.~Lian, M.~Yan, C.~Zhang, and J.~Liu,
\newblock ``{D}$^2$: {D}ecentralized training over decentralized data,''
\newblock in {\em Proceedings of the 35th International Conference on Machine
  Learning}, Jul. 2018, vol.~80, pp. 4848--4856.

\bibitem{harnessing}
G.~Qu and N.~Li,
\newblock ``Harnessing smoothness to accelerate distributed optimization,''
\newblock {\em IEEE Transactions on Control of Network Systems}, vol. 5, no. 3,
  pp. 1245--1260, 2017.

\bibitem{AB_proc}
R.~Xin, S.~Pu, A.~Nedi{\'c}, and U.~A. Khan,
\newblock ``A general framework for decentralized optimization with first-order
  methods,''
\newblock {\em Proceedings of the IEEE}, vol. 108, no. 11, pp. 1869--1889,
  2020.

\bibitem{GT_SAGA_SPM}
R.~Xin, S.~Kar, and U.~A. Khan,
\newblock ``Decentralized stochastic optimization and machine learning,''
\newblock {\em IEEE Signal Processing Magazine}, May 2020.

\bibitem{DAVRG}
K.~Yuan, B.~Ying, J.~Liu, and A.~H. Sayed,
\newblock ``Variance-reduced stochastic learning by networked agents under
  random reshuffling,''
\newblock {\em IEEE Transactions on Signal Processing}, vol. 67, no. 2, pp.
  351--366, 2018.

\bibitem{pushsaga}
M.~I. Qureshi, R.~Xin, S.~Kar, and U.~A. Khan,
\newblock ``Push-saga: A decentralized stochastic algorithm with variance
  reduction over directed graphs,''
\newblock {\em IEEE Control Systems Letters}, vol. 6, pp. 1202--1207, 2022.

\bibitem{GT_SARAH_ncvx}
R.~Xin, U.~A. Khan, and S.~Kar,
\newblock ``A near-optimal stochastic gradient method for decentralized
  non-convex finite-sum optimization,''
\newblock {\em arXiv:2008.07428}, 2020.

\bibitem{GT_HSGD_ncvx}
R.~Xin, U.~A. Khan, and S.~Kar,
\newblock ``A hybrid variance-reduced method for decentralized stochastic
  non-convex optimization,''
\newblock {\em arXiv:2102.06752}, 2021.

\bibitem{xin2021stochastic}
R.~Xin, S.~Das, U.~A. Khan, and S.~Kar,
\newblock ``A stochastic proximal gradient framework for decentralized
  non-convex composite optimization: {T}opology-independent sample complexity
  and communication efficiency,'' 2021,
\newblock arXiv: 2110.01594.

\bibitem{WW_GDA_mingyi}
H.~Wai, Z.~Yang, Z.~Wang, and M.~Hong,
\newblock ``Multi-agent reinforcement learning via double averaging primal-dual
  optimization,''
\newblock in {\em Proceedings of the 32nd International Conference on Neural
  Information Processing Systems}, Red Hook, NY, USA, 2018, NIPS'18, p.
  9672–9683, Curran Associates Inc.

\bibitem{Local_GDA}
Y.~Deng and M.~Mahdavi,
\newblock ``Local stochastic gradient descent ascent: Convergence analysis and
  communication efficiency,'' 2021.

\bibitem{AB_GDA}
J.~Ren, J.~Haupt, and Z.~Guo,
\newblock ``Communication-efficient hierarchical distributed optimization for
  multi-agent policy evaluation,''
\newblock {\em Journal of Computational Science}, vol. 49, pp. 101280, 2021.

\bibitem{rockafellar}
R.~T. Rockafellar,
\newblock {\em Convex analysis},
\newblock Princeton Mathematical Series. Princeton University Press, Princeton,
  N. J., 1970.

\bibitem{Nesterov_book}
Y.~Nesterov,
\newblock {\em Lectures on convex optimization}, vol. 137,
\newblock Springer, 2018.

\bibitem{DGD_nedich}
A.~Nedi\'c and A.~Ozdaglar,
\newblock ``Distributed subgradient methods for multi-agent optimization,''
\newblock {\em IEEE Trans. on Autom. Control}, vol. 54, no. 1, pp. 48, 2009.

\bibitem{conj_sm_sc}
S.~M. Kakade and S.~Shalev-Shwartz,
\newblock ``On the duality of strong convexity and strong smoothness : Learning
  applications and matrix regularization,''
\newblock 2009.

\bibitem{Schmidt}
M.~W. Schmidt, G.~Fung, and R.~Rosales,
\newblock ``Fast optimization methods for l1 regularization: A comparative
  study and two new approaches,''
\newblock in {\em ECML}, 2007.

\bibitem{hornjohnson}
R.~A. Horn and C.~R. Johnson,
\newblock {\em Matrix Analysis},
\newblock Cambridge University Press, Cambridge, 2nd edition, 2012.

\bibitem{ABSAGA}
M.~I. Qureshi, R.~Xin, S.~Kar, and U.~A. Khan,
\newblock ``Variance reduced stochastic optimization over directed graphs with
  row and column stochastic weights,'' 2022,
\newblock arXiv: 2202.03346.

\bibitem{polyak}
B.~Polyak,
\newblock {\em Introduction to optimization},
\newblock Optimization Software, 1987.

\bibitem{pushsaga_icassp}
M.~I. Qureshi, R.~Xin, S.~Kar, and U.~A. Khan,
\newblock ``A decentralized variance-reduced method for stochastic optimization
  over directed graphs,''
\newblock in {\em 2021 IEEE International Conference on Acoustics, Speech and
  Signal Processing (ICASSP)}, 2021, pp. 5030--5034.

\bibitem{semi_simp_eig}
A.P. Seyranian and A.A. Mailybaev,
\newblock {\em Multiparameter Stability Theory With Mechanical Applications},
\newblock World Scientific Publishing Company, 2003.

\bibitem{DSVM_Reza}
M.~Doostmohammadian, A.~Aghasi, T.~Charalambous, and U.~A. Khan,
\newblock ``Distributed support vector machines over dynamic balanced directed
  networks,''
\newblock {\em IEEE Control Systems Letters}, vol. 6, pp. 758--763, 2022.

\bibitem{kai}
K.~Cai and H.~Ishii,
\newblock ``Average consensus on general strongly connected digraphs,''
\newblock {\em Automatica}, vol. 48, 03 2012.

\end{thebibliography}
\bibliographystyle{IEEEbib}

\appendix
\subsection{Proof of Lemma 1}\label{p_lem1}
For convenience, we restate some notation. The vector-matrix form of~$\GDA$,~where we define global vectors~${\mb x^k,\mb q^k, \nabla_x \mb f^k}$, all in~$\mbb R^{n p_x}$, and~${\mb y^k,\mb w^k, \nabla_y \mb f^k}$, all in~$\mbb R^{n p_y}$, i.e.,
\begin{align*}
\mb{x}^k &:= \left[ {\begin{array}{c}
    \mb{x}^{k}_1\\
    \vdots\\
    \mb{x}^{k}_n
  \end{array} } \right], \quad \mb{q}^k := \left[ {\begin{array}{c}
    \mb{q}^{k}_1\\
    \vdots\\
    \mb{q}^{k}_n
  \end{array} } \right], \quad \nabla_x \mb f^k := \left[ {\begin{array}{c}
    \nabla_x \mb f^k_1\\
    \vdots\\
    \nabla_x \mb f^k_n
  \end{array} } \right],
\end{align*}
% asas
\begin{align*}
  \mb{y}^k &:= \left[ {\begin{array}{c}
    \mb{y}^{k}_1\\
    \vdots\\
    \mb{y}^{k}_n
  \end{array} } \right], \quad \mb{w}^k := \left[ {\begin{array}{c}
    \mb{w}^{k}_1\\
    \vdots\\
    \mb{w}^{k}_n,
  \end{array} } \right], \quad \nabla_y \mb f^k := \left[ {\begin{array}{c}
    \nabla_y \mb f^k_1\\
    \vdots\\
    \nabla_y \mb f^k_n
  \end{array} } \right],
\end{align*}
and the global matrices as~${W_1:=W \otimes I_{p_x} \in \R^{np_x \times np_x}}$, ${W_2:=W \otimes I_{p_y} \in \R^{np_y \times np_y}}$ and
\begin{align*}
P^k &:= \left[ {\begin{array}{c}
    P^{k}_1\\
    \vdots\\
    P^{k}_n
  \end{array} } \right].
\end{align*}
$\GDA$ in Algorithm~\ref{algo} can be equivalently written as
\begin{align*}
{P}^{k+1} &= W_2 P^{k},\\
\mb{x}^{k+1} &= W_1 (\mb{x}^{k} - \alpha  \cdot \mb{q}^{k}),\\
\mb{y}^{k+1} &= W_2 (\mb{y}^{k} + \beta  \cdot \mb{w}^{k}),\\
\mb{q}^{k+1} &= W_1 (\mb{q}^{k} + \nabla_x \mb f^{k+1} - \nabla_x \mb f^{k}),\\
\mb{w}^{k+1} &= W_2 (\mb{w}^{k} + \nabla_y \mb f^{k+1} - \nabla_y \mb f^{k}).
\end{align*}
% \end{subequations}
To aid the analysis, we re-define three error quantities:
\begin{enumerate}[(i)]
    \item Agreement errors~${\| \mb{x}^k - W_1^\infty \mb{x}^k \|}$ and~${\| \mb{y}^k - W_2^\infty \mb{y}^k \|}$;
    \item Optimality gaps~${\| \ol{\mb{x}}^k - \mb{x}^* \|}$ and~${\| \ol{\mb{y}}^k - \nabla H^*(\ol{P} \ol{\mb{x}}^k) \|}$;
    \item Gradient tracking errors~${\| \mb{q}^k - W_1^\infty \mb{q}^k \|^2}$ and ${\| \mb{w}^k - W_2^\infty \mb{w}^k \|^2}$.
\end{enumerate} 
Next we calculate error bounds on these quantities to establish system dynamics for $\GDA$.
\subsubsection{Agreement errors}
The network agreement error for $\mb{x}$ and $\mb{y}$ updates can be quantified as
\begin{align*}
    &~~~~\|\mb{x}^{k+1} - W_1^\infty \mb{x}^{k+1}\| \\
    &= \|W_1 (\mb{x}^{k} - \alpha  \cdot \mb{q}^{k}) - W_1^\infty (W_1 (\mb{x}^{k} - \alpha  \cdot \mb{q}^{k}))\| \\
    &= \|W_1(\mb{x}^{k} - W_1^\infty \mb{x}^{k}) - \alpha W_1 (\mb{q}^{k} - W_1^\infty \mb{q}^{k})\| \\
    &\leq \lambda \|\mb{x}^{k} - W_1^\infty \mb{x}^{k}\| + \alpha  \|\mb{q}^{k} - W_1^\infty \mb{q}^{k}\|,
\end{align*}
where we use that~${W_1^\infty W_1 = W_1^\infty}$ in the second step. Moreover, we use~${\|W_1 \mb x^k - W_1^\infty \mb x^k \| \leq \lambda \|\mb x^k - W^\infty \mb x^k \|}$, such that~${\lambda := \|W - W^\infty\|< 1}$, and triangular inequality in the third step. Similarly, we have
\begin{align*}
    \|\mb{y}^{k+1} - W_2^\infty \mb{y}^{k+1}\| &\leq \lambda\|\mb{y}^{k} - W_2^\infty \mb{y}^{k}\| + \beta \|\mb{w}^{k} - W_2^\infty \mb{w}^{k}\|.
\end{align*}
\subsubsection{Optimality gaps}
We first consider the bound on $\|\ol{\mb{x}}^{k+1} - \mb{x}^*\|$. It can be verified that for~${\mb{x}^+:=\ol{\mb{x}}^{k} - \alpha \nabla \theta (\ol{\mb{x}}^{k})}$, 
\begin{align*}
    \|\ol{\mb{x}}^{k+1} - \mb{x}^*\| &\leq \|\ol{\mb{x}}^{k+1} - \mb{x}^{+}\| + \|\mb{x}^{+} - \mb{x}^*\| \\
    &\leq \|\ol{\mb{x}}^{k+1} - \mb{x}^{+}\| + \left(1-\alpha \frac{\sigma^2_m}{L_2} \right)\|\ol{\mb{x}}^{k} - \mb{x}^*\|,
\end{align*}
where~$\mb{x}^+$ is a gradient descent step for the primal problem for global objective~${\theta(\mb{x}):= R(\mb{x}) + G(\mb{x})}$, such that~${R(\mb{x}):= H^*(\ol{P} \mb{x})}$, and we used the~${\tfrac{\sigma_m^2}{L_2}}$-strong convexity of $\theta$, see \cite{du} for details. Next we find an upper bound on the first term as follows
\begin{align*}
    &\|\ol{\mb{x}}^{k+1} - \mb{x}^{+}\|
    \leq \alpha \|\nabla \theta (\ol{\mb{x}}^{k}) - \ol{\nabla_x \mb{f}}^k\| \\ 
    &\leq \alpha \frac{L_1}{\sqrt{n}} \|\mb{x}^k - W_1^\infty \mb{x}^k\| + \alpha \| \frac{\sum_{i=1}^n (P_i^k)^\top \mb{y}_i}{n}  - \ol{P}^\top \nabla H^* (\ol{P}\ol{\mb{x}}^k)\| \\
    &\leq \alpha \frac{L_1}{\sqrt{n}} \|\mb{x}^k - W_1^\infty {\mb{x}}^k\| + \alpha \sigma_M \| \ol{\mb{y}}^k  - \nabla H^* (\ol{P}\ol{\mb{x}}^k)\| \\
    &~~~+ \alpha \frac{1}{\sqrt{n}} \lambda^k \tau \| \mb{y}^k\|,
\end{align*}
where for~${\tau:= \mn{P^0 - \mb{1}_n \otimes \ol{P}}}$ we used
\begin{align*}
    &\| \frac{\sum_{i=1}^n (P_i^k)^\top \mb{y}_i}{n}  - \ol{P}^\top \nabla H^* (\ol{P}\ol{\mb{x}}^k)\| \\
    &\leq \| \frac{\sum_{i=1}^n \left[ (P_i^k)^\top - \ol{P}^{\top} \right] \mb{y}_i^k }{n}\| + \|\ol{P}^\top \ol{\mb{y}}^k - \ol{P}^\top \nabla H^*(\ol{P} \ol{\mb{x}}^k) \| \\
    &\leq \frac{1}{\sqrt{n}} \lambda^k \mn{(P^0)^\top - \mb{1}_n \otimes \ol{P}^{\top}} \| \mb{y}^k \| + \sigma_M \|\ol{\mb{y}}^k - \nabla H^*(\ol{P} \ol{\mb{x}}^k) \|.
\end{align*}
Plugging this in the above expression leads to
\begin{align*}
    \|\ol{\mb{x}}^{k+1} - \mb{x}^*\| &\leq \alpha \frac{L_1}{\sqrt{n}} \|\mb{x}^k - W_1^\infty {\mb{x}}^k\| + \alpha \sigma_M \|\ol{\mb{y}}^{k} - \nabla H^* (\ol{P} \ol{\mb{x}}^{k})\| \\
    &~~~+ \left(1-\alpha \frac{\sigma^2_m}{L_2} \right)\|\ol{\mb{x}}^{k} - \mb{x}^*\| + \alpha \frac{1}{\sqrt{n}} \lambda^k \tau \|\mb{y}^k \|.
\end{align*}
Similarly we would like to evaluate the upper bound on~${\|\ol{\mb{y}}^{k+1} - \nabla H^* (\ol{P} \ol{\mb{x}}^{k+1})\|}$ but first it is significant to note that for~${\mb{y}^* :=\nabla H^* (\ol{P} \mb{x}^*)}$ we have
\begin{align*}
    \|\ol{\mb{y}}^{k} - \mb{y}^*\| &= \|\ol{\mb{y}}^{k} - \nabla H^* (\ol{P} \mb{x}^*)\| \\
    % &\leq \|\ol{\mb{y}}^{k} - \nabla H^* ({P} \ol{\mb{x}}^k)\| + \|\nabla H^* ({P} \ol{\mb{x}}^k) - \nabla H^* ({P} {\mb{x}}^*)\| \\
    &\leq \|\ol{\mb{y}}^{k} - \nabla H^* (\ol{P} \ol{\mb{x}}^k)\| + \frac{\sigma_M}{{\mu}} \|\ol{\mb{x}}^k - \mb{x}^*\|,
\end{align*}
where we used~${\mn{P} \leq \sigma_M}$ and~${1/\mu}$-strong convexity of $H^*$. Therefore, it is sufficient to evaluate the upper bound on~${\|\ol{\mb{y}}^{k+1} - \nabla H^* ({P} \ol{\mb{x}}^{k+1})\|}$ to establish the optimality gap. Thus, we expand
\begin{align*}
    &\|\ol{\mb{y}}^{k+1} - \nabla H^* (\ol{P} \ol{\mb{x}}^{k+1})\| \\
    &= \|\ol{\mb{y}}^{k+1} - \nabla H^* (\ol{P} \ol{\mb{x}}^{k}) + \nabla H^* (\ol{P} \ol{\mb{x}}^{k}) - \nabla H^* (\ol{P} \ol{\mb{x}}^{k+1})\| \\
    % &\leq \|\ol{\mb{y}}^{k} + \beta \ol{\nabla_y \mb{f}} - \nabla H^* (\ol{P} \ol{\mb{x}}^{k})\| + \|\nabla H^* (\ol{P} \ol{\mb{x}}^{k}) - \nabla H^* (\ol{P} \ol{\mb{x}}^{k+1})\| \\
    % &= \|\ol{\mb{y}}^{k} - \beta \nabla \wh{H}(\ol{\mb{y}}^k) + \beta \nabla \wh{H}(\ol{\mb{y}}^k) + \beta \ol{\nabla_y \mb{f}} - \nabla H^* (\ol{P} \ol{\mb{x}}^{k})\| + \|\nabla H^* (\ol{P} \ol{\mb{x}}^{k}) - \nabla H^* (\ol{P} \ol{\mb{x}}^{k+1})\| \\
    &\leq \|\ol{\mb{y}}^{k} - \beta \nabla \wh{H}(\ol{\mb{y}}^k) - \nabla H^* (\ol{P} \ol{\mb{x}}^{k})\| + \beta \| \nabla \wh{H}(\ol{\mb{y}}^k) + \ol{\nabla_y \mb{f}}\| \\
    &~~~+ \|\nabla H^* (\ol{P} \ol{\mb{x}}^{k}) - \nabla H^* (\ol{P} \ol{\mb{x}}^{k+1})\| \\
    &\leq (1-\beta {\mu})\|\ol{\mb{y}}^{k} - \nabla H^* (\ol{P} \ol{\mb{x}}^{k})\| + \beta \frac{L_2}{\sqrt{n}} \| \mb{y}^k - W_2^\infty \mb{y}^k\| \\
    &~~~+ \beta \frac{1}{\sqrt{n}} \lambda^k \tau \|\mb{x}^k \| + \frac{\sigma_M}{{\mu}} \| \ol{\mb{x}}^{k} - \ol{\mb{x}}^{k+1}\|,
\end{align*}
where~${\ol{\mb{y}}^{k+1} = \ol{\mb{y}}^{k} - \beta \left(\nabla H(\ol{\mb{y}}^k) - \ol{P} \ol{\mb{x}}^k\right)}$ is the gradient descent step for ${\mu}$-strongly convex and ${L_2}$-smooth objective function defined as ${\wh{H}({\mb{y}}) := H({\mb{y}}) - \langle{\mb{y}}, \ol{P} \ol{\mb{x}}^k \rangle}$. By optimality condition we have the gradient equal to zero, the minimizer satisfies $\wh{\mb{y}} = \nabla H^*(\ol{P} \ol{\mb{x}}^k)$ and the first term follows. The last term can be further expanded as:
\begin{align*}
\|\ol{\mb{x}}^{k} - \ol{\mb{x}}^{k+1}\| &\leq \alpha \| \ol{\nabla_x \mb{f}}^k - \nabla_x F(\ol{\mb{x}}^k, \ol{\mb{y}}^k) \| \\
&~~~+ \alpha \| \nabla_x F(\ol{\mb{x}}^k, \ol{\mb{y}}^k) - \nabla_x F({\mb{x}}^*, {\mb{y}}^*) \| \\
&\leq \alpha \left( \frac{L}{\sqrt{n}} \|\mb{x}^k - W_1^\infty \mb{x}^k \| + \frac{\sigma_M}{\sqrt{n}} \|\mb{y}^k - W_2^\infty \mb{y}^k\| \right) \\
&~~~+ \alpha \left( L \|\ol{\mb{x}}^{k} - \mb{x}^*\| + \sigma_M \|\ol{\mb{y}}^{k} - \mb{y}^*\| \right).
\end{align*}
Let~${\alpha \leq \beta \frac{{\mu}^2}{c \sigma_M^2}}$, where~${c>1}$. Using the above, we get the final expression for the error bound on~${\|\ol{\mb{y}}^{k+1} - \nabla H^* ({P} \ol{\mb{x}}^{k+1})\|}$:
\begin{align*}
    &\|\ol{\mb{y}}^{k+1} - \nabla H^* (\ol{P} \ol{\mb{x}}^{k+1})\| \\
    &\leq \left[1-\beta {\mu}\left(1 - \frac{1}{c} \right)  \right]\|\ol{\mb{y}}^{k} - \nabla H^* (\ol{P} \ol{\mb{x}}^{k})\| \\
    &~~~+ \frac{\beta}{c} \left(\frac{L_1 {\mu}}{\sigma_M} + \sigma_M\right) \|\ol{\mb{x}}^{k} - \mb{x}^*\| + \alpha \frac{\sigma_M L_1}{{\mu} \sqrt{n}} \| \mb{x}^k - {W}_1^\infty \mb{x}^k\| \\
    &~~~+ \left(\beta \frac{{L_2}}{\sqrt{n}} + \alpha \frac{\sigma_M^2}{{\mu}\sqrt{n}}\right) \|\mb{y}^k - {W}_2^\infty \mb{y}^k\| + \beta \frac{1}{\sqrt{n}} \lambda^k \tau \| \mb{x}^k \|.
\end{align*}
\subsubsection{Gradient tracking}
Finally we evaluate the upper bounds on gradient tracking errors for descent and ascent updates
\begin{align*}
    &\|\mb{q}^{k+1} - W_1^\infty \mb{q}^{k+1}\| \\
    % &= \|W_1(\mb{q}^{k} + \nabla_x \mb{f}^{k+1} - \nabla_x \mb{f}^{k}) - W_1^\infty W_1(\mb{q}^{k} + \nabla_x \mb{f}^{k+1} - \nabla_x \mb{f}^{k})\| \\
    % &= \|W_1(\mb{q}^{k} - W_1^\infty \mb{q}^{k}) + W_1( \nabla_x \mb{f}^{k+1} - \nabla_x \mb{f}^{k}) - W_1^\infty W_1(\nabla_x \mb{f}^{k+1} - \nabla_x \mb{f}^{k})\| \\
    &= \|W_1(\mb{q}^{k} - W_1^\infty \mb{q}^{k}) + W_1(I - W_1^\infty)( \nabla_x \mb{f}^{k+1} - \nabla_x \mb{f}^{k})\| \\
    &\leq \lambda \|\mb{q}^{k} - W_1^\infty \mb{q}^{k}\| + \lambda \|\nabla_x \mb{f}^{k+1} - \nabla_x \mb{f}^{k}\|,
\end{align*}
where we used the triangular inequality and the facts that~${\mn{W_1} = 1}$ and~${\mn{W_1 - W^\infty_1} = \lambda}$. Next we expand on the second term and it can be verified that:
\begin{align*}
&\| \nabla_x \mb{f}^{k+1} - \nabla_x \mb{f}^{k} \| \\
&= \sqrt{\sum_{i=1}^n \| \nabla_x g_i (\mb{x}_i^{k+1}) - \nabla_x g_i (\mb{x}_i^{k}) + P_i^\top (\mb{y}_i^{k+1} - \mb{y}_i^k)\|^2} \\
&\leq \sqrt{\sum_{i=1}^n L^2 \|\mb{x}_i^{k+1} - \mb{x}_i^k\|^2} + \sqrt{\sum_{i=1}^n \sigma_M^2 \|\mb{y}_i^{k+1} - \mb{y}_i^k\|^2}\\
&\leq \left(L + \alpha L^2 + \beta \sigma_M^2 \right) \|\mb{x}^{k} - W^\infty_1 \mb{x}^{k}\| + \alpha L \|\mb{q}^k - W^\infty_1 \mb{q}^k\| \\
&~~~+ (\alpha L^2 + \beta \sigma_M^2) \sqrt{n} \|\ol{\mb{x}}^{k} - \mb{x}^*\| + \beta \sigma_M \|\mb{w}^k - W^\infty_2 \mb{w}^k\| \\
&~~~+ (\alpha \sigma_M L + \beta \sigma_M L_2)\sqrt{n} \|\ol{\mb{y}}^{k} - \mb{y}^*\| \\
&~~~+ (\sigma_M + \beta \sigma_M L_2 + \alpha \sigma_M L) \|\mb{y}^{k} - W^\infty_2 \mb{y}^{k}\|
\end{align*}
where we used 
\begin{align*}
\| \mb{x}^{k} - \mb{x}^{k+1}\| &\leq \| \mb{x}^{k} - W^\infty_1 \mb{x}^{k} \| + \alpha \|\mb{q}^k \|,
\end{align*}
and
\begin{align*}
\|\mb{q}^k\| &\leq \|\mb{q}^k - W^\infty_1 \mb{q}^k\| + \sqrt{n} \|\ol{\nabla_x \mb{f}}^k - \nabla_x F(\ol{\mb{x}}^k, \ol{\mb{y}}^k)\| \\
&~~~+ \sqrt{n} \|\nabla_x F(\ol{\mb{x}}^k, \ol{\mb{y}}^k) - \nabla_x F(\mb{x}^*, \mb{y}^*) \| \\
&\leq \|\mb{q}^k - W^\infty_1 \mb{q}^k\| + L \|\mb{x}^k - W^\infty_1 \mb{x}^k \| \\
&~~~+ \sigma_M \|\mb{y}^k - W^\infty_2 \mb{y}^k\| + L \sqrt{n} \|\ol{\mb{x}}^{k} - \mb{x}^*\| \\
&~~~+ \sigma_M \sqrt{n} \|\ol{\mb{y}}^{k} - \mb{y}^*\|.
\end{align*}
Using the above bounds, we get the final expression for~${L^{-1}\|\mb{q}^{k+1} - W_\infty \mb{q}^{k+1}\|}$ as follows:
\begin{align*}
&L^{-1} \|\mb{q}^{k+1} - W_\infty \mb{q}^{k+1}\| \\
&= \lambda \left(1 + \alpha L + \beta \frac{\sigma_M^2}{L} \right) \|\mb{x}^{k} - W^\infty_1 \mb{x}^{k}\| \\
&~~~+ \lambda (1 + \alpha L) L^{-1} \|\mb{q}^k - W^\infty_1 \mb{q}^k\| \\
&~~~+ \lambda\left(\alpha \left(L + \frac{\sigma_M^2}{{\mu}}\right) + \beta \frac{\sigma_M^2}{L} \left(1 + \kappa \right)\right) \sqrt{n} \|\ol{\mb{x}}^{k} - \mb{x}^*\| \\
&~~~+ \lambda \left(\alpha \sigma_M + \beta \sigma_M \right)\sqrt{n} \|\ol{\mb{y}}^{k} - \nabla H^* (\ol P \ol{\mb{x}}^k)\| \\
&~~~+ \lambda \left(\frac{\sigma_M}{L} + \beta \sigma_M + \alpha \sigma_M \right) \|\mb{y}^{k} - W^\infty_2 \mb{y}^{k}\| \\
&~~~+ \lambda \beta \sigma_M L^{-1}\|\mb{w}^k - W^\infty_2 \mb{w}^k\|.
\end{align*}
Using similar arguments, it can be verified that the gradient tracking error~${L}^{-1} \|\mb{w}^{k+1} - W^\infty_2 \mb{w}^{k+1}\|$ can be expressed as follows:
\begin{align*}
&{L}^{-1} \|\mb{w}^{k+1} - W^\infty_2 \mb{w}^{k+1}\| \\
&\leq \lambda \left(1 + \beta {L} + \alpha \frac{\sigma_M^2}{{L}}\right)\|\mb{y}^{k} - W^\infty_2 \mb{y}^{k}\| \\
&~~~+ \lambda \left(1 + \beta {L}\right) {L}^{-1}\|\mb{w}^k - W^\infty_2 \mb{w}^k\| \\
&~~~+ \lambda \left[\beta \sigma_M \left(1 + \kappa \right) + \alpha \frac{\sigma_M}{{L}} \left(L + \frac{\sigma_M^2}{{\mu}}\right)\right] \sqrt{n} \|\ol{\mb{x}}^{k} - \mb{x}^*\| \\
&~~~+ \lambda \left(\beta {L} + \alpha \frac{\sigma_M^2}{{L}} \right) \sqrt{n} \|\ol{\mb{y}}^{k} - \nabla H^* (\ol{P} \ol{\mb{x}}^k)\| \\
&~~~+ \lambda \left(\frac{\sigma_M}{{L}} + \alpha \sigma_M + \beta \sigma_M \right) \|\mb{x}^{k} - W^\infty_1 \mb{x}^{k}\| \\
&~~~+ \lambda \alpha \sigma_M {L}^{-1}\|\mb{q}^k - W^\infty_1 \mb{q}^k\|,
\end{align*}
and Lemma~\ref{lem1} follows.

\subsection{System matrix for~$\GDA$}\label{appA}
To completely describe Lemma~\ref{lem1}, we define the system matrix as follows
\begin{align*}
  M_{\alpha,\beta} &:= \left[ {\begin{array}{c c}
    M_{11} &M_{12} \\
    M_{21} &M_{22}
  \end{array} } \right],
\end{align*}
{\small \begin{align*}
  M_{11} &:= \left[ {\begin{array}{c c c}
    \lambda &0 &\alpha L\\
    \alpha L_1 &1-\alpha \frac{\sigma^2_m}{{L_2}} &0 \\
    \lambda + \alpha \lambda L + \beta \frac{\lambda \sigma_M^2}{L} & \alpha m_1 + \beta m_2 &\lambda + \alpha \lambda L
  \end{array} } \right], \\
  M_{12} &:= \left[ {\begin{array}{c c c}
    0 &0 &0 \\
    0 &\alpha &0 \\
    \frac{\lambda}{L} + \alpha \lambda + \beta \lambda &\alpha \lambda + \beta \lambda &\beta \lambda
  \end{array} } \right] \sigma_M, \\
  M_{21} &:= \left[ {\begin{array}{c c c}
    0 &0 &0 \\
    \alpha \frac{\sigma_M L_1}{{\mu}} &\alpha m_3 &0 \\
    \frac{\lambda \sigma_M}{{L}} + \alpha \lambda \sigma_M + \beta \lambda \sigma_M &\alpha m_4 + \beta m_5 &\alpha \lambda \sigma_M
  \end{array} } \right], \\
  M_{22} &:= \left[ {\begin{array}{c c c}
    \lambda &0 &\beta {L}\\
    \alpha \frac{\sigma_M^2}{{\mu}} + \beta {L_2} &1 - \beta {\mu}\left( 1 - \frac{1}{c}\right) &0 \\
    \lambda + \alpha \frac{\lambda \sigma_M^2}{{L}} + \beta \lambda {L} & \alpha \frac{\lambda \sigma_M^2}{{L}} + \beta \lambda {L} &\lambda + \beta \lambda L
  \end{array} } \right],
\end{align*}}where~${c > \frac{2 L^2}{\sigma_m^2} + \frac{2 \sigma_M^2 {\kappa}}{\sigma_m^2} + 1 }$ and 
\begin{align*}
    m_1 &:= \lambda \left(L + \tfrac{\sigma_M^2}{{\mu}}\right), \qquad ~~~~m_2 := \lambda \tfrac{\sigma_M^2}{L} \left(1 + \tfrac{{L}}{{\mu}} \right),\\
    m_3 &:= \tfrac{\sigma_M}{\mu} \left( L_1 + \tfrac{\sigma_M^2}{\mu} \right), \qquad m_4 := \lambda \tfrac{\sigma_M}{{L}} \left(L + \tfrac{\sigma_M^2}{{\mu}} \right), \\
    m_5 &:= \lambda \sigma_M \left(1 + \tfrac{{L}}{{\mu}} \right).
\end{align*}

\subsection{LTI system for~$\GDAl$}\label{sys_gdal}
While deriving the LTI system for~$\GDAl$ under Assumptions~\ref{sm_scc},~\ref{cp_rank}, and~\ref{doub_stoc}, it can be verified that the only differences occur while evaluating the bounds on optimality gaps. Thus, we only write the two updated inequalities:
\begin{align*}
    \|\ol{\mb{x}}^{k+1} - \mb{x}^*\| &\leq \alpha \frac{L_1}{\sqrt{n}} \|\mb{x}^k - W_1^{\infty} {\mb{x}}^k\| + \alpha \sigma_M \|\ol{\mb{y}}^{k} - \nabla H^* (\ol{P} \ol{\mb{x}}^{k})\| \\
    &~~~+ \left(1-\alpha \frac{\sigma^2_m}{L_2} \right)\|\ol{\mb{x}}^{k} - \mb{x}^*\| + \alpha \frac{1}{\sqrt{n}}  \tau \|\mb{y}^k \|,
\end{align*}
\begin{align*}
    &\|\ol{\mb{y}}^{k+1} - \nabla H^* (\ol{P} \ol{\mb{x}}^{k+1})\| \\
    &\leq \left[1-\beta {\mu}\left(1 - \frac{1}{c} \right)  \right]\|\ol{\mb{y}}^{k} - \nabla H^* (\ol{P} \ol{\mb{x}}^{k})\| \\
    &~~~+ \frac{\beta}{c} \left(\frac{L_1 {\mu}}{\sigma_M} + \sigma_M\right) \|\ol{\mb{x}}^{k} - \mb{x}^*\| + \alpha \frac{\sigma_M L_1}{{\mu} \sqrt{n}} \| \mb{x}^k - {W}_1^\infty \mb{x}^k\| \\
    &~~~+ \left(\beta \frac{{L_2}}{\sqrt{n}} + \alpha \frac{\sigma_M^2}{{\mu}\sqrt{n}}\right) \|\mb{y}^k - {W}_2^\infty \mb{y}^k\| + \beta \frac{1}{\sqrt{n}} \tau \| \mb{x}^k \|,
\end{align*}
and Theorem~$\ref{th2}$ follows.
% \newpage
\subsection{System matrix for~$\GDAl$: Quadratic case}\label{appB}

To completely describe Lemma~\ref{gdal_quad}, we define the system matrix as follows
% Let~${p := p_x + p_y}$, and we define~${\wt{\mb{u}}^{k} \in \R^{(2n+1)p}}$, and~${\wt{M}_{\alpha, \beta} \in \R^{(2n+1)p \mt (2n+1)p}}$ as
% {\small \begin{align*}
% \wt{\mb{u}}^{k} \!&:=\!  \left[ \begin{array}{c}
%     \mb{x}^{k} - W_1^\infty \mb{x}^{k} \\
%     \ol{\mb{x}}^{k} - \mb{x}^* \\
%     \mb{q}^{k} - W_1^\infty \mb{q}^{k} \\
%     \mb{y}^{k} - W_2^\infty \mb{y}^{k} \\
%     \ol{\mb{y}}^{k} - \mb{y}^* \\
%     \mb{w}^{k} - W_2^\infty \mb{w}^{k}
%   \end{array}\right], \qquad
%   \wt{M}_{\alpha, \beta}:= \left[ {\begin{array}{c c}
%     \wt{M}_{11} &\wt{M}_{12} \\
%     \wt{M}_{21} &\wt{M}_{22}
%   \end{array} } \right],
% \end{align*}}
\begin{align*}
  \wt{M}_{\alpha, \beta}:= \left[ {\begin{array}{c c}
    \wt{M}_{11} &\wt{M}_{12} \\
    \wt{M}_{21} &\wt{M}_{22}
  \end{array} } \right],
\end{align*}
% \begin{align*}
%   \wt{M}_{11} &:= \left[ {\begin{array}{c c c}
%     \ol{W}_1 &0 &-\alpha W_1\\
%     - \alpha \frac{(\mb{1}_n^\top \otimes I_{p_x}) \Lambda_Q}{n} &I - \alpha \ol{Q} &0 \\
%     \wt{m}_1 &\wt{m}_2 &\wt{m}_3
%   \end{array} } \right], \\
%   \wt{M}_{12} &:= \left[ {\begin{array}{c c c}
%     0 &0 &0 \\
%     -\alpha \frac{(\mb{1}_n^\top \otimes I_{p_x}) \Lambda_{P^\top}}{n} &-\alpha \ol{P}^\top &0 \\
%     \wt{m}_4 &\wt{m}_5 &\wt{m}_6
%   \end{array} } \right], \\
%   \wt{M}_{21} &:= \left[ {\begin{array}{c c c}
%     0 &0 &0 \\
%     \beta \frac{(\mb{1}_n^\top \otimes I_{p_y}) \Lambda_P}{n} &\beta \ol{P} &0 \\
%     \ol{m}_1 &\ol{m}_2 &\ol{m}_3
%   \end{array} } \right],
% \end{align*}
\begin{align*}
  \wt{M}_{11} &:= \left[ {\begin{array}{c c c}
    \ol{W}_1 &O &-\alpha W_1\\
    - \alpha \frac{(\mb{1}_n^\top \otimes I_{p_x}) \Lambda_Q}{n} &I_{p_x} - \alpha \ol{Q} &O \\
    \wt{m}_1 &\wt{m}_2 &\wt{m}_3
  \end{array} } \right], \\
  \wt{M}_{12} &:= \left[ {\begin{array}{c c c}
    O &O &O \\
    -\alpha \frac{(\mb{1}_n^\top \otimes I_{p_x}) \Lambda_{P^\top}}{n} &-\alpha \ol{P}^\top &O \\
    \wt{m}_4 &\wt{m}_5 &\wt{m}_6
  \end{array} } \right], \\
  \wt{M}_{21} &:= \left[ {\begin{array}{c c c}
    O &O &O \\
    \beta \frac{(\mb{1}_n^\top \otimes I_{p_y}) \Lambda_P}{n} &\beta \ol{P} &O \\
    \ol{m}_1 &\ol{m}_2 &\ol{m}_3
  \end{array} } \right], \\
  \wt{M}_{22} &:= \left[ {\begin{array}{c c c}
    \ol{W} &O &\beta W_2\\
    -\beta \frac{(\mb{1}_n^\top \otimes I_{p_y}) \Lambda_{R}}{n} &I_{p_y} - \beta \ol{R} &O \\
    \ol{m}_4 &\ol{m}_5 &\ol{m}_6
  \end{array} } \right],
\end{align*}
such that for any collection of matrices~${\{M_1, M_2, \cdots, M_n\}}$, we define~${\Lambda_{M}:=\mbox{diag}([M_1, M_2, \cdots, M_n])}$ as the block diagonal matrix where the~$i$-th diagonal element is~$M_i$. For~${\ol{W}_1:= W_1 - W_1^\infty}$,~${\ol{W}_2:= W_2 - W_2^\infty}$, other terms used in~${M_{\alpha, \beta}}$ are defined below
\begin{align*}
    \wt{m}_1 &= \ol{W}_1 \left[\Lambda_Q \left((W_1 - I_{n p_x}) - \alpha W^\infty_1  \Lambda_Q \right) + \beta \Lambda_{P^\top} W_2^\infty \Lambda_{P}\right], \\
    \wt{m}_2 &= \ol{W}_1 \left[-\alpha \Lambda_Q W_1^\infty \Lambda_Q + \beta \Lambda_{P^\top} W_2^\infty \Lambda_{P}\right] (\mb{1}_n\otimes I_{p_x}), \\
    \wt{m}_3 &= \left[\ol{W}_1 - \alpha \Lambda_Q W_1 \right], \\
    \wt{m}_4 &= \ol{W}_1 \left[-\alpha \Lambda_Q W_1^\infty \Lambda_{P^\top} + \Lambda_{P^\top} \left((W_2 - I_{n p_y}) - \beta W_2^\infty \Lambda_{R} \right) \right], \\
    \wt{m}_5 &= \ol{W}_1 \left[-\alpha \Lambda_Q W_1^\infty \Lambda_{P^\top} - \beta \Lambda_{P^\top} W_2^\infty \Lambda_{R}\right] (\mb{1}_n\otimes I_{p_y}), \\
    \wt{m}_6 &= \ol{W}_1 \left[\beta \Lambda_{P^\top} W_2 \right],\\
    \ol{m}_1 &= \ol{W}_2 \left[- \beta \Lambda_{R} W_2^\infty \Lambda_P + \Lambda_P \left((W_1-I_{n p_x})-\alpha W_1^\infty \Lambda_Q \right) \right], \\
    \ol{m}_2 &= \ol{W}_2 \left[- \beta \Lambda_{R} W_2^\infty \Lambda_P - \alpha \Lambda_P W_1^\infty \Lambda_Q\right] (\mb{1}_n\otimes I_{p_x}), \\
    \ol{m}_3 &= \ol{W}_2 \left[-\alpha \Lambda_P W_1 \right],\\
    \ol{m}_4 &= \ol{W}_2 \left[- \Lambda_{R} \left((W_2 - I_{n p_y}) - \beta W_2^\infty \Lambda_{R} \right) -\alpha \Lambda_P W^\infty_1 \Lambda_{P^\top} \right], \\
    \ol{m}_5 &= \ol{W}_2 \left[\beta \Lambda_{R} W_2^\infty \Lambda_{R} - \alpha \Lambda_P W_1^\infty \Lambda_{P^\top} \right] (\mb{1}_n\otimes I_{p_y}), \\
    \ol{m}_6 &= \left[\ol{W}_2 - \beta \Lambda_R W_2 \right].
\end{align*}

\end{document}